\documentclass[11pt,leqno]{article}
\usepackage{amsmath,amsthm}
\usepackage{amsfonts}
\usepackage{amssymb}
\usepackage{epsfig}
\usepackage{subfig}
\usepackage{graphicx}
\usepackage{bm}
\numberwithin{equation}{section} 
\newtheorem{theorem}{\bf Theorem}[section]
\newtheorem{example}{\bf Example}[section]

\newtheorem{remark}{\bf Remark}[section]
\newtheorem{lemma}{\bf Lemma}[section]

\newtheorem{proposition}{\bf Proposition}[section]

\newcommand{\norm}[1]{\left\lVert #1\right\rVert}

\newsavebox{\savepar}

\begin{document}
\title{ On Kirchhoff's Model of Parabolic Type}
\author{ Sudeep Kundu, Amiya K. Pani\\
Department of Mathematics\\
IIT Bombay, Powai, Mumbai-400076 (India)\\
Email: sudeep.kundu85@gmail.com, akp@math.iitb.ac.in\\
and\\
Morrakot Khebchareon\\
Department of Mathematics, Faculty of Science\\ 
Chiang Mai University, Chiang Mai-50200 (Thailand)\\
Email:{mkhebcha@gmail.com}.}

\maketitle
\abstract{ In this paper,
existence of a  strong global solution for all finite time is derived for the 
Kirchhoff's model of parabolic type. 
Based on exponential weight function, some new regularity results which reflect the exponential decay property 
are obtained for the exact solution. For the related dynamics, existence of a global attractor is shown to hold for 
the problem, when the non-homogeneous  forcing function is either 
independent of time or in  $L^{\infty} (L^2).$  With finite element Galerkin method applied in spatial
direction keeping time variable continuous, a semidiscrete scheme is analyzed and it is, further, established
that the semi-discrete system has a global discrete attractor.
 Optimal error estimates in $L^\infty(H^1_0)$-norm are derived which are valid uniformly in time. Further, 
 based on a Backward Euler method, a completely discrete scheme is developed and error estimates are 
 derived. It is further observed that in case $f=0$ or $f=O(e^{-\gamma_0 t})$ with $\gamma_0>0$, the discrete solutions and also error estimates decay exponentially. Finally, some numerical experiments are discussed 
 which confirm our theoretical findings.}
\section{Inroduction.}
In this article, we consider the following nonlocal nonlinear boundary value problem of Kirchhoff's model of parabolic type: 
Find $u=u(x,t)$, $x\in \Omega$ and $t>0$ which satisfies 
 \begin{align}
  u_t-\left(1+\norm{\nabla u(t)}^2_{L^2(\Omega)}\right)\Delta u &=f \qquad\text{in }\Omega\times(0,\infty)\label{eq1},\\
  u(x,t) &=0 \qquad\text{on }\partial\Omega\times(0,\infty)\label{eq2},\\
  u(x,0) &=u_0 \qquad\text{in }\Omega\label{eq3},
 \end{align} 
 where $\Omega$ is a convex polygonal or polyhedral domain in $\mathbb{R}^d(d=2\hspace{0.1cm}\text{or}\hspace{0.1cm} 3)$ with boundary $\partial\Omega$, $f=f(x,t)$ and $u_0$ are given functions in 
 their respective domain of definitions. Here, $u_t=\frac{\partial u}{\partial t}$ and $\norm{.}_{L^2(\Omega)}$ the $L^2$-norm. Such problem arises in the model describing the evolution of a population density 
 subjected to a diffusion rate proportional to $(1+\norm{\nabla u(t)}^2)$ with   
 the forcing function $f$ representing the rate of supply.  
For details of the physical application and its complete mathematical modelling of such type of problems, 
we refer to \cite{chipot0} and \cite{Alves}.
More general nonlinear parabolic equations with nonlocal terms is of the form
 \begin{equation}\label{extra1}
  u_t-a(\norm{\nabla u}^2_{L^2(\Omega)})\Delta u=f(x,t), \quad (x,t)\in\Omega\times(0,\infty),
 \end{equation}
 where $a$ is a nonlinear nonlocal form in $u$, which includes the problem \eqref{eq1} are considered in the literature, see 
  \cite{chipot1}, \cite{chipot2}, \cite{chipot3} and \cite{chipot4}. In these articles, the focus is on proving well-posedness and on the study 
  of asymptotic behavior of solutions of the nonlocal problem \eqref{extra1}-\eqref{eq3} under various conditions on the nonlinearity.

 In recent years, numerical approximation to the stationary problem of \eqref{eq1}-\eqref{eq3} has been studied in \cite{gudi} using $C^0$-conforming finite element method and optimal error estimates in $H^1$ are derived. 
 However, there is hardly any result on  the numerical approximations to 
\eqref{eq1}-\eqref{eq3}. When the forcing function $f$ is either independent of time or is in $L^\infty(L^2)$, 
it plays a crucial role in the dynamics of this problem, therefore, in this paper,  
global existence of a unique strong  solution to the problem \eqref{eq1}-\eqref{eq3} for all $t\in [0,T]$ with any finite positive $T>0$ is proved using Bubnov Galerkin method and compactness arguments. New regularity results 
using exponential weight function are also established. As a consequence,  this problem admits the 
existence of a global attractor both in $L^2$ and $H^1_0$-spaces.
When $C^0$- conforming finite element method is applied to approximate the solution of \eqref{eq1} 
keeping time variable continuous, a semidiscrete scheme is derived and it is shown that the discrete problem 
has a discrete global attractor. Further, optimal  {\it {priori}} error estimates in $L^{\infty}(H^1)$-norm
are established which are even valid uniformly in time. Then 
based on backward Euler method, a discrete scheme is analyzed and it is, further, shown that 
the discrete problem has a solution using a variant of Brouwer fixed point argument. Moreover,
optimal error estimates are derived.  When either $f=0$ or $f=O(e^{-\gamma_0 t})$ with some 
$\gamma_0 >0,$ exponential decay property for the exact as well as  the discrete solution and for 
error estimates is shown to hold. Finally some numerical experiments are conducted which confirm our theoretical
results.

The main contributions of this article are to
\begin{itemize}
\item derive regularity results using exponential weight functions and establish
global existence of a unique strong solution to problem \eqref{eq1}-\eqref{eq3}.
\item prove optimal error estimates of the  semidiscrete Galerkin approximation, which are valid 
uniformly in time and with right kind of regularity for  the problem with convex polygonal or polyhedral domains.
\item show the existence of a global attractor in both continuous and semidiscrete cases.
\item prove exponential decay property for the exact solution, discrete solution and even for the error
when the forcing function is either zero or of decaying exponentially in time.
\item provide error analysis for the completely discrete scheme which is based on backward Euler method without
using discrete Gronwall's inequality.
\end{itemize}
%When the forcing function is either independent of time or in $L^{\infty}(L^2),$ exponential 

The rest of the article is organized as follows. Section $2$ is devoted to the global existence and uniqueness of strong solution and new a {\it{priori}} bounds for the solutions of \eqref{eq1}-\eqref{eq3} are derived. Section $3$ deals with error estimates for the semidiscrete solutions. Section $4$, focuses on the existence and uniqueness of the discrete solution and error estimates. Finally, in section $5$ some numerical results 
are discussed to confirm theoretical results.

\section{Global Existence, Uniqueness and Regularity Results}
 This section deals with weak formulation, existence of unique global solution and some regularity results. Denote by $H^m(\Omega)$ as the standard Sobolev spaces with norm $\norm{.}_m.$ \\
 Set ${H^1_0(\Omega)}=\{v\in H^1(\Omega): v=0 \quad \text{on} \quad \partial\Omega\}$. Let $H^{-1}$ be the dual space of $H^1_0(\Omega)$.
The space $L^p([0,T];X)$  $1\leq p\leq\infty, $ consists of all strongly measurable functions $v:[0,T] \rightarrow X $ with norm
 $$\norm{v}_{L^p([0,T];X)}:=\left(\int_{0}^{T}\norm{v(t)}^p dt\right)^\frac{1}{p}<\infty \quad \text{for} \quad 1\leq p<\infty,$$ and 
% $$\norm{u}_{L^\infty([0,T];X)}:=\operatorname{ess}\sup\limits_{0\leq t\leq T}\norm{u(t)}$$
 $$\norm{v}_{L^\infty([0,T];X)}:=\operatorname*{ess\,sup}\limits_{0\leq t\leq T}\norm{v(t)}<\infty.$$
 The weak formulation related to the problem \eqref{eq1}-\eqref{eq3} is to seek $v(t)\in H^1_0(\Omega), $ $t>0$ such that 
 \begin{equation}\label{temp1}
 (u_t,v)+\left(1+\norm{\nabla u}^2_{L^2(\Omega)}\right)(\nabla u,\nabla v)=(f,v) \qquad \forall v\in H^1_0(\Omega)\qquad\text{with }  u(0)=u_0.
 \end{equation}
 \subsection{A priori bounds}
 This subsection focuses on {\it{a priori}} bounds for the problem \eqref{temp1} which are valid uniformly in time
  using exponential weight functions in time.
\begin{lemma}\label{2.1}
 Assume that $f\in L^\infty(H^{-1})$ and $u_0\in L^2(\Omega).$ Then, there holds for $ 0<\alpha<\dfrac{\lambda_1}{2}$
  \begin{align}
   \norm{u(t)}^2+\beta& e^{-2\alpha t}\int_{0}^{t} e^{2\alpha s}\norm{\nabla u(s)}^2 ds
   +2e^{-2\alpha t}\int_{0}^{t} e^{2\alpha s}\norm{\nabla u(s)}^4ds\notag\\
   &\leq e^{-2\alpha t}\norm{u_0}^2+\frac{1}{2\alpha}\norm{f}^2_{L^\infty(H^{-1})}(1-e^{-2\alpha t})=\widehat{K}_0(t)\label{temp2}\\
   &\leq\norm{u_0}^2+\frac{1}{2\alpha}\norm{f}^2_{L^\infty(H^{-1})}=K_0,\notag
  \end{align}
  where $\beta=(1-\dfrac{2\alpha}{\lambda_1})>0$, and $\lambda_1>0$ is the minimum eigenvalue of the Dirichlet eigenvalue problem for the Laplace operator.
\end{lemma}
\begin{proof}
 Set $v=e^{2\alpha t}u$ in \eqref{temp1} and obtain
  \begin{align*}
 \frac{1}{2}\frac{d}{dt}(e^{2\alpha t}\norm{u(t)}^2)-\alpha e^{2\alpha t}\norm{u(t)}^2+e^{2\alpha t}(1+\norm{\nabla u(t)}^2)\norm{\nabla u}^2
 &\leq e^{2 \alpha t}\norm{f}_{H^{-1}}\norm{\nabla u}\\
 &\leq\frac{1}{2}e^{2\alpha t}\norm{f}^2_{-1}+\frac{1}{2}e^{2\alpha t}\norm{\nabla u}^2.
  \end{align*}
  Apply Poincare's inequality : $\norm{\varphi}\leq\dfrac{1}{\sqrt{\lambda_1}}\norm{\nabla\varphi}$ for $\varphi \in H^1_0$, where $\lambda_1$ is 
  the minimum eigenvalue of the Laplace operator with Dirichlet boundary condition. Thus, on integration with respect to time and using 
  $ab\leq \frac{a^2}{2}+\frac{b^2}{2}$, we obtain
   \begin{align*}
    e^{2\alpha t}\norm{u(t)}^2+ \left(1-\dfrac{2\alpha}{\lambda_1}\right)\int_{0}^{t}e^{2\alpha s}\norm{\nabla u(s)}^2 ds
    &+2\int_{0}^{t}e^{2\alpha s}\norm{\nabla u(s)}^4ds \\
    &\leq \norm{u_0}^2+\frac{1}{2\alpha}(e^{2\alpha t}-1)\norm{f}^2_{L^\infty(H^{-1})}.
 \end{align*}
 Since $\alpha$ can be chosen so that $0<\alpha<\dfrac{\lambda_1}{2}$ and $\beta=(1-\dfrac{2\alpha}{\lambda_1})>0$. Then, it follows that
 \begin{align*}
   \norm{u(t)}^2+\beta e^{-2\alpha t}\int_{0}^{t} e^{2\alpha s}\norm{\nabla u(s)}^2 ds
  & +2e^{-2\alpha t}\int_{0}^{t} e^{2\alpha s}\norm{\nabla u(s)}^4ds\\
   &\leq e^{-2\alpha t}\norm{u_0}^2+\frac{1}{2\alpha}\norm{f}^2_{L^\infty(H^{-1})}(1-e^{-2\alpha t}),
 \end{align*}
 and this concludes the rest of the proof.
\end{proof}
\begin{remark}
 If $f\in{L^\infty(L^2)}$, then rewrite using Poincare's inequality
\begin{align*}
e^{2\alpha t}(f,u)\leq e^{2\alpha t}\norm{f}_{L^\infty(L^2)}\norm{u}&\leq \dfrac{e^{2\alpha t}}{\sqrt{\lambda_1}}\norm{f}_{L^\infty(L^2)}\norm{\nabla u}\\
&\leq \dfrac{e^{2\alpha t}}{2\lambda_1}\norm{f}^2_{L^\infty(L^2)}+\frac{1}{2}\norm{\nabla u}^2
 \end{align*}
Following the part of Lemma \ref{2.1}, it now follows that 
\begin{equation}\label{rm1}
 \norm{u(t)}\leq e^{-2\alpha t}\norm{u_0}^2+\dfrac{1}{2\lambda_1\alpha}\norm{f}^2_{L^\infty(L^2)}(1-e^{-2\alpha t}).
\end{equation}
\end{remark}

\begin{lemma}\label{2.2}
 Let $u_0\in H^1_0(\Omega)$ and $f\in L^\infty(L^2)$. Then for $0<\alpha<\dfrac{\lambda_1}{2}$, there holds:
 \begin{align}
  \norm{\nabla u(t)}^2+\beta& e^{-2\alpha t}\int_{0}^{t} e^{2\alpha s}\norm{\Delta u(s)}^2 ds
  +2e^{-2\alpha t}\int_{0}^{t}e^{2\alpha s}\norm{\nabla u(s)}^2\norm{\Delta u(s)}^2ds\notag\\
  &\leq e^{-2\alpha t}\norm{\nabla u_0}^2+\frac{1}{2\alpha}\norm{f}^2_{L^\infty(L^2)}(1-e^{-2\alpha t})=\widehat K_1(t)\label{temp3}\\
  &\leq\norm{\nabla u_0}^2+\frac{1}{2\alpha}\norm{f}^2_{L^\infty(L^2)}=K_1.\notag
 \end{align}
\end{lemma}
\begin{proof}
Forming $L^2$- inner product between \eqref{eq1} and $-e^{2\alpha t} \Delta u$ yields 
 \begin{align*}
   \frac{1}{2}\frac{d}{dt}(e^{2\alpha t}\norm{\nabla u(t)}^2)-\alpha e^{2\alpha t}\norm{\nabla u(t)}^2
   & +(1+\norm{\nabla u}^2)e^{2\alpha t}\norm{\Delta u(t)}^2\\
   & =-e^{2\alpha t}(f,\Delta u)\\
   &\leq \frac{1}{2}e^{2\alpha t}\norm{f}^2+\frac{1}{2} e^{2\alpha t}\norm{\Delta u}^2.
 \end{align*}
Using Poincare's inequality, it follows using integration with respect to time that
\begin{align*}
 e^{2\alpha t}\norm{\nabla u(t)}^2+(1-\dfrac{2\alpha}{\lambda_1})\int_{0}^{t} e^{2\alpha s}\norm{\Delta u(s)}^2 ds
 & +2\int_{0}^{t}e^{2\alpha s}\norm{\nabla u(s)}^2\norm{\Delta u(s)}^2ds\\
 &\leq \norm{\nabla u_0}^2+\frac{1}{2\alpha}\norm{f}^2_{L^\infty(L^2)}(e^{2\alpha t}-1).
\end{align*}
Multiplying by $e^{-2\alpha t}$ and with $\beta=(1-\dfrac{2\alpha}{\lambda_1})>0, $ it completes the rest of the proof.
\end{proof}

\begin{lemma}\label{2.3}
  Let $u_0\in H^1_0(\Omega)$ and $f\in L^\infty(L^2)$. Then for $0<\alpha<\dfrac{\lambda_1}{2}$, there holds:
 \begin{align*}
   2e^{-2\alpha t}\int_{0}^{t}&e^{2\alpha s}\norm{u_s}^2ds+(2+\norm{\nabla u(t)}^2)\norm{\nabla u(t)}^2\\
  &\leq(2+\norm{\nabla u_0}^2)\norm{\nabla u_0}^2e^{-2\alpha t}+\frac{1}{\alpha}\norm{f}^2_{L^\infty(L^2)}(1-e^{-2\alpha t})
  +\alpha(1+\dfrac{4}{\beta})\widehat K_0(t)=\widehat K_3(t)\\
  &\leq(2+\norm{\nabla u_0}^2)\norm{\nabla u_0}^2+\frac{1}{\alpha}\norm{f}^2_{L^\infty(L^2)}+\alpha(1+\dfrac{4}{\beta})K_0=K_3.
 \end{align*}
\end{lemma}
\begin{proof}
 Set $v=u_t$ in \eqref{temp1} and obtain 
\begin{equation*}
 \norm{u_t}^2+\frac{1}{4}\frac{d}{dt}((2+\norm{\nabla u}^2)\norm{\nabla u}^2)\leq \frac{1}{2}\norm{f}^2+\frac{1}{2}\norm{u_t}^2,
\end{equation*}
and hence, 
\begin{equation*}
 \norm{u_t}^2+\frac{1}{2}\frac{d}{dt}(2+\norm{\nabla u}^2)\norm{\nabla u}^2\leq \norm{f}^2.
\end{equation*}
Multiplying by $2e^{2\alpha t}$ and then, rewrite it as 
\begin{equation*}
 2e^{2\alpha t}\norm{u_t}^2+\frac{d}{dt}\left(e^{2\alpha t}(2+\norm{\nabla u}^2)\norm{\nabla u}^2\right)\leq2e^{2\alpha t}\norm{f}^2+2\alpha
 e^{2\alpha t}(2+\norm{\nabla u}^2)\norm{\nabla u}^2.
\end{equation*}
Thus, on integration with respect time from $0$ to $t$  and then, multiplying the resulting inequality by $e^{-2\alpha t}$ to obtain
\begin{align*}
 2e^{-2\alpha t}\int_{0}^{t} e^{2\alpha s}\norm{u_s}^2ds+&(2+\norm{\nabla u}^2)\norm{\nabla u}^2\\&\leq(2+\norm{\nabla u_0}^2)\norm{\nabla u_0}^2
 e^{-2\alpha t}+\dfrac{1}{\alpha}\norm{f}^2_{L^\infty(L^2)}(1-e^{-2\alpha t})\\
 &+2\alpha e^{-2\alpha t}\int_{0}^{t}e^{2\alpha s}(2+\norm{\nabla u}^2)\norm{\nabla u}^2 ds.
\end{align*}
Using Lemma \ref{temp1}, we note that
\begin{align*}
 2\alpha e^{-2\alpha t}\int_{0}^{t}e^{2\alpha s}(2+\norm{\nabla u}^2)\norm{\nabla u}^2 ds \leq\alpha\left(1+\dfrac{4}{\beta}\right)\widehat K_0(t).
\end{align*}
Thus, we arrive at
\begin{align*}
  2e^{-2\alpha t}\int_{0}^{t} e^{2\alpha s}\norm{u_s}^2ds+&(2+\norm{\nabla u(t)}^2)\norm{\nabla u(t)}^2\\
  &\leq(2+\norm{\nabla u_0}^2)\norm{\nabla u_0}^2e^{-2\alpha t}+\dfrac{1}{\alpha}\norm{f}^2_{L^\infty(L^2)}(1-e^{-2\alpha t})\\
  &+\alpha\left(1+\dfrac{4}{\beta}\right)\widehat K_0(t)=\widehat K_3(t)\\
  &\leq(2+\norm{\nabla u_0}^2)\norm{\nabla u_0}^2+\frac{1}{\alpha}\norm{f}^2_{L^\infty(L^2)}+\alpha\left(1+\dfrac{4}{\beta}\right)K_0=K_3.
\end{align*}
This concludes the rest of the proof.
\end{proof}
\begin{lemma}\label{2.4}
   Let $u_0\in H^2\cap H^1_0(\Omega)$, $f\in L^\infty(L^2)$ and $f_t\in L^\infty(H^{-1})$. Then there holds
 \begin{align*}
  &\norm{u_t(t)}^2+e^{-2\alpha t}\int_{0}^{t}e^{2\alpha s}(\beta+2\norm{\nabla u}^2)\norm{\nabla u_s}^2ds+4e^{-2\alpha t}\int_{0}^{t}e^{2\alpha s}(\nabla u_s,\nabla u)^2ds\\
  &\leq e^{-2\alpha t}\left((1+\norm{\nabla u_0}^2)\norm{u_0}_{H^2}+\norm{f_0}\right)^2+e^{-2\alpha t}\int_{0}^{t}e^{2\alpha s}\norm{f_s}^2_{-1}\;ds.
 \end{align*}
 \end{lemma}
 \begin{proof}
  Differentiating of \eqref{temp1} with respect to time yields
  \begin{align}
   (u_{tt},v)+\left((1+\norm{\nabla u}^2)\nabla u_t,\nabla v\right)+\left(2(\nabla u_t,\nabla u)\nabla u,\nabla v\right)=(f_t,v).\label{eq2.4}
  \end{align}
 Substitute $v=u_t$ in \eqref{eq2.4} to obtain
 \begin{align}
  \frac{1}{2}\frac{d}{dt}\norm{u_t}^2+(1+\norm{\nabla u}^2)\norm{\nabla u_t}^2+2
  (\nabla u_t,\nabla u)^2=(f_t,u_t).\label{eq2.5}
 \end{align}
 Multiplying \eqref{eq2.5} by $2e^{2\alpha t}$, $\alpha>0$ and using Poincare's inequality it follows that
 \begin{align*}
  \frac{d}{dt}(e^{2\alpha t}\norm{u_t}^2)+e^{2\alpha t}\left(1-\frac{2\alpha}{\lambda_1}\right)\norm{\nabla u_t}^2+2e^{2\alpha t}(1+\norm{\nabla u}^2)\norm{\nabla u_t}^2+&4e^{2\alpha t}(\nabla u_t,\nabla u)^2\\
  &\leq e^{2\alpha t}\norm{f_t}^2_{-1}.
 \end{align*}
 Integrating with respect to time from $0$ to $t$ and multiplying the resulting inequality by $e^{-2\alpha t}$ to obtain
 \begin{align}
  &\norm{u_t(t)}^2+e^{-2\alpha t}\int_{0}^{t}e^{2\alpha s}(\beta+2\norm{\nabla u}^2)\norm{\nabla u_s}^2ds+4e^{-2\alpha t}\int_{0}^{t}e^{2\alpha s}(\nabla u_s,\nabla u)^2ds\notag\\
  &\leq e^{-2\alpha t}\norm{u_t(0)}^2+e^{-2\alpha t}\int_{0}^{t}e^{2\alpha s}\norm{f_s}^2_{-1}ds\label{eq2.6}\\
  &\leq e^{-2\alpha t}\left((1+\norm{\nabla u_0}^2)\norm{u_0}_{H^2}+\norm{f_0}\right)^2+e^{-2\alpha t}\int_{0}^{t}e^{2\alpha s}\norm{f_s}^2_{-1} ds.\notag
 \end{align}
This completes the rest of the proof.
 \end{proof}
 \begin{lemma}\label{2.5}
  Let $u_0\in H^2\cap H^1_0(\Omega)$, $f\in L^\infty(L^2)$ and $f_t\in L^\infty(H^{-1})$. Then, there holds:
\begin{align*}
(1+\norm{\nabla u}^2)\norm{\Delta u}^2\leq2\norm{f}^2_{L^\infty(L^2)}+2e^{-2\alpha t}\Big(\left((1+\norm{\nabla u_0}^2)\norm{u_0}_{H^2}+\norm{f_0}\right)^2&+\\
 \int_{0}^{t}e^{2\alpha s}\norm{f_s}^2_{-1}\;ds\Big).
\end{align*}
\end{lemma}
\begin{proof}
 Substitute $v=-\Delta u$ in the Weak formulation \eqref{temp1} to obtain 
 \begin{align*}
  (1+\norm{\nabla u}^2)\norm{\Delta u}^2=-(f,\Delta u)+(u_t,\Delta u).
 \end{align*}
 Using Young's inequality, we bound
 \begin{align*}
  (1+\norm{\nabla u}^2)\norm{\Delta u}^2&\leq\norm{f}^2_{L^\infty(L^2)}+\norm{u_t}^2_{L^\infty(L^2)}+\frac{1}{2}\norm{\Delta u}^2\\
  &\leq\norm{f}^2_{L^\infty(L^2)}+\norm{u_t}^2_{L^\infty(L^2)}+\frac{1}{2}(1+\norm{\nabla u}^2)\norm{\Delta u}^2 .                                       
 \end{align*}
 Therefore, we arrive at
 \begin{align}
  (1+\norm{\nabla u}^2)\norm{\Delta u}^2\leq2\norm{f}^2_{L^\infty(L^2)}+2\norm{u_t}^2_{L^\infty(L^2)}.
 \end{align}
 From Lemma \ref{2.4} applying the bound of $\norm{u_t}$ we obtain bound for $\norm{\Delta u}.$ This 
 completes the proof.
\end{proof}
% In order to esimate $\norm{\Delta u}$ take $L^2$ inner product between \eqref{eq1} and $-\Delta u_t$
  Since $\Omega$ is a convex polygonal bounded domain, hence $\norm{u}^2_{H^2}\leq C_R\norm{\Delta u}^2$. Now from  Lemmas \ref{2.4}-\ref{2.5}, it follows that
$$\norm{u(t)}^2_{H^2}\leq2C_R\left(\norm{f}^2_{L^\infty(L^2)}+e^{-2\alpha t}\left((1+\norm{\nabla u_0}^2)^2\norm{u_0}^2_{H^2}+\norm{f_0}^2\right)+e^{-2\alpha t}\int_{0}^{t}e^{2\alpha s}\norm{f_s}^2_{-1}ds\right).$$
%  Since $\Omega$ is a convex polygonal bounded domain, hence $\norm{u}^2_{H^2}\leq C_R\norm{\Delta u}^2$. Hence, from  Lemmas \ref{2.4} , \ref{2.5}, it follows that
%  $$\norm{u(t)}^2_{H^2}\leq\frac{C_R(K_3)}{\tau}\left(\norm{f}^2_{L^\infty(L^2)}+e^{-2\alpha t}\int_{0}^{t}e^{2\alpha s}\norm{f_s}^2_{-1}ds\right)$$
 \begin{lemma}\label{nl1}
  Let $u_0\in H^3$, $f\in L^\infty(L^2)$ and $f_t\in L^\infty(H^{-1})$. Then, there holds
  \begin{align*}
  \norm{\nabla u_t}^2+&e^{-2\alpha t}\int_{0}^{t}e^{2\alpha s}(1+\norm{\nabla u}^2)\norm{\Delta u_t(s)}^2 ds\notag\\
  &\leq C\Big(e^{-2\alpha t}\norm{u_0}^2_{H^3}+e^{-2\alpha t}\left((1+\norm{\nabla u_0}^2)^2\norm{u_0}^2_{H^2}+\norm{f_0}^2\right)\notag\\
  &+e^{-2\alpha t}\int_{0}^{t}e^{2\alpha s}\norm{f_s}^2_{-1}ds\Big).
%     \norm{\nabla u_t}^2+&e^{-2\alpha t}\int_{0}^{t}e^{2\alpha s}\Big((1+\norm{\nabla u}^2)(1-2\dfrac{\alpha}{\lambda_1})-\norm{\Delta u}^2\Big)\norm{\Delta u_t(s)}^2 ds\\
%      &\leq e^{-2\alpha t}\norm{u_0}^2_{H^3}+e^{-2\alpha t}\left((1+\norm{\nabla u_0}^2)\norm{u_0}_{H^2}+\norm{f_0}\right)+2e^{-2\alpha t}\int_{0}^{t}e^{2\alpha s}\norm{f_s}^2_{-1}ds.
  \end{align*}

 \end{lemma}
\begin{proof}
  Differentiating of \eqref{temp1} with respect to time yields
  \begin{align}
   (u_{tt},v)-\left((1+\norm{\nabla u}^2)\Delta u_t,v\right)-\left(2(\nabla u_t,\nabla u)\Delta u,v\right)=(f_t,v).\label{ne1}
  \end{align}
  Substitute $v=-\Delta u_t$ in \eqref{ne1} to obtain
 \begin{align}
  \frac{d}{dt}\norm{\nabla u_t}^2+2(1+\norm{\nabla u}^2)\norm{\Delta u_t}^2=2(f_t,-\Delta u_t)+4\Big((u_t,-\Delta u)\Delta u,-\Delta u_t\Big).\label{ne2}
 \end{align}
 On multiplying \eqref{ne2} by $e^{2\alpha t},\alpha>0$ and using Young's inequality, it follows that
 \begin{align}
  \frac{d}{dt}(e^{2\alpha t}\norm{\nabla u_t}^2)&+2e^{2\alpha t}(1+\norm{\nabla u}^2)\norm{\Delta u_t}^2\notag\\
  &\leq 2e^{2\alpha t}\norm{f_t}^2+\frac{1}{2}e^{2\alpha t}\norm{\Delta u_t}^2+8e^{2\alpha t}\norm{u_t}^2\norm{\Delta u}^4\notag\\
  &+\frac{1}{2} e^{2\alpha t}\norm{\Delta u_t}^2+2\alpha e^{2\alpha t}\norm{\nabla u_t}^2.\label{ne3}
 \end{align}
Now using Poincare's inequality, we find that
\begin{align}
 \frac{d}{dt}(e^{2\alpha t}\norm{\nabla u_t}^2)&+e^{2\alpha t}(1+\norm{\nabla u}^2)\norm{\Delta u_t}^2\notag\\
 &\leq 2e^{2\alpha t}\norm{f_t}^2+8e^{2\alpha t}\norm{u_t}^2\norm{\Delta u}^4+2\alpha e^{2\alpha t}\norm{\nabla u_t}^2.
\end{align}
 Integrating with respect to time from $0$ to $t$ and multiplying the resulting inequality by $e^{-2\alpha t}$ to obtain
 \begin{align}
  \norm{\nabla u_t}^2+e^{-2\alpha t}\int_{0}^{t}e^{2\alpha s}&(1+\norm{\nabla u}^2)\norm{\Delta u_t(s)}^2 ds\notag\\
  &\leq e^{-2\alpha t}\norm{\nabla u_t(0)}^2+2e^{-2\alpha t}\int_{0}^{t}e^{2\alpha s}\norm{f_s}^2 ds\notag\\
  &+8e^{-2\alpha t}\int_{0}^{t}e^{2\alpha s}\norm{u_t}^2\norm{\Delta u}^4 ds+2\alpha e^{-2\alpha t}\int_{0}^{t}e^{2\alpha s}\norm{\nabla u_t(s)}^2 ds.\label{ne4}
 \end{align}
Therefore by Lemma \ref{2.4}, we arrive at
\begin{align}
  \norm{\nabla u_t}^2+&e^{-2\alpha t}\int_{0}^{t}e^{2\alpha s}(1+\norm{\nabla u}^2)\norm{\Delta u_t(s)}^2 ds\notag\\
  &\leq C\Big(e^{-2\alpha t}\norm{u_0}^2_{H^3}+e^{-2\alpha t}\left((1+\norm{\nabla u_0}^2)^2\norm{u_0}^2_{H^2}+\norm{f_0}^2\right)\notag\\
  &+e^{-2\alpha t}\int_{0}^{t}e^{2\alpha s}\norm{f_s}^2_{-1}ds\Big).
\end{align}
This completes the rest of the proof.
\end{proof}
% --------------------------------------
\begin{lemma}\label{nl2}
 Let $u_0\in H^3$, $f\in L^\infty(L^2)$ and $f_t\in L^\infty(H^{-1})$. Then, the following result holds
 \begin{align*}
   e^{-2\alpha t}\int_{0}^{t}&e^{2\alpha s}\norm{u_{tt}(s)}^2 ds+(1+\norm{\nabla u(t)}^2)\norm{\nabla u_t}^2\\
  &\leq e^{-2\alpha t}\norm{u_0}^2_{H^3}+C(K_1)\Big(e^{-2\alpha t}\left((1+\norm{\nabla u_0}^2)^2\norm{u_0}^2_{H^2}+\norm{f_0}^2\right)+e^{-2\alpha t}\int_{0}^{t}e^{2\alpha s}\norm{f_s}^2_{-1}ds\Big).
 \end{align*}
\end{lemma}
\begin{proof}
Differentiating of \eqref{temp1} with respect to time yields
  \begin{align}
 (u_{tt},v)+\left((1+\norm{\nabla u}^2)\nabla u_t,\nabla v\right)+\left(2(\nabla u_t,\nabla u)\nabla u,\nabla v\right)=(f_t,v).\label{ne5}
 \end{align}
 Substitute $v=u_{tt}$ in \eqref{ne5} to obtain
 \begin{align}
  \norm{u_{tt}}^2+\frac{1}{2}(1+\norm{\nabla u}^2)\frac{d}{dt}\norm{\nabla u_t}^2-2\Big((\nabla u_t,\nabla u)\Delta u,u_{tt}\Big)=(f_t,u_{tt}).\label{ne6}
 \end{align}
On multiplying \eqref{ne6} by $e^{2\alpha t},\alpha>0$ and rewriting it as
\begin{align}
  e^{2\alpha t}\norm{u_{tt}}^2+\frac{1}{2}\frac{d}{dt}(e^{2\alpha t}(1+\norm{\nabla u}^2)\norm{\nabla u_t}^2) &=e^{2\alpha t}(f_t,u_{tt})+2e^{2\alpha t}\Big((\nabla u_t,\nabla u)\Delta u,u_{tt}\Big)\notag\\
  &+e^{2\alpha t}(\nabla u_t,\nabla u)\norm{\nabla u_t}^2 +\alpha e^{2\alpha t}(1+\norm{\nabla u}^2)\norm{\nabla u_t}^2.\label{ne7}
\end{align}
Now using Young's inequality for the first two terms in the right hand side and rewriting it as
\begin{align}
 e^{2\alpha t}\norm{u_{tt}}^2+&\frac{d}{dt}(e^{2\alpha t}(1+\norm{\nabla u}^2)\norm{\nabla u_t}^2)\notag\\
 &\leq 2e^{2\alpha t}\norm{f_t}^2+e^{2\alpha t}\Big(8\norm{\nabla u}^2\norm{\Delta u}^2+2\norm{\nabla u}\norm{\nabla u_t}+2\alpha(1+\norm{\nabla u}^2)\Big)\norm{\nabla u_t}^2.\label{ne8}
\end{align}
 Integrating with respect to time from $0$ to $t$ and multiplying the resulting inequality by $e^{-2\alpha t}$ to obtain
 \begin{align}
   e^{-2\alpha t}\int_{0}^{t}&e^{2\alpha s}\norm{u_{tt}(s)}^2 ds+(1+\norm{\nabla u(t)}^2)\norm{\nabla u_t}^2\notag\\
   &\leq e^{-2\alpha t}(1+\norm{\nabla u(0)}^2)\norm{\nabla u_t(0)}^2+2e^{-2\alpha t}\int_{0}^{t}e^{2\alpha s}\norm{f_s}^2 ds\notag\\
   &+e^{-2\alpha t}\int_{0}^{t}e^{2\alpha s}\Big(8\norm{\nabla u}^2\norm{\Delta u}^2+2\norm{\nabla u}\norm{\nabla u_t(s)}+2\alpha(1+\norm{\nabla u}^2)\Big)\norm{\nabla u_t(s)}^2 ds.\label{ne9}
 \end{align}
 The term inside the bracket of the third term in the right hand side is bounded by Lemmas \ref{2.2}, \ref{2.5} and \ref{nl1}. Therefore the third term in the right hand side is bounded by the lemma \ref{2.4}.
Altogether, we obtain
 \begin{align}
  e^{-2\alpha t}\int_{0}^{t}&e^{2\alpha s}\norm{u_{tt}(s)}^2 ds+(1+\norm{\nabla u(t)}^2)\norm{\nabla u_t}^2\notag\\
  &\leq e^{-2\alpha t}\norm{u_0}^2_{H^3}+C(K_1)\Big(e^{-2\alpha t}\left((1+\norm{\nabla u_0}^2)^2\norm{u_0}^2_{H^2}+\norm{f_0}^2\right)+e^{-2\alpha t}\int_{0}^{t}e^{2\alpha s}\norm{f_s}^2_{-1}ds\Big).\label{ne10}
 \end{align}
 This completes the rest of the proof.
 \end{proof}
% -----------------------------
\begin{remark}
 \begin{enumerate}
 \item When $f=0$, then we obtain $$\norm{u(t)},\norm{\nabla u(t)}=O(e^{-\alpha t})\quad \text{and} \quad{\tau}^\frac{1}{2}(\norm{u_t(t)}+\norm{u(t)}_{H^2})=O(e^{-\alpha t}).$$
 Hence, we derive exponential decay property.
 
 \item When $f\in L^\infty(L^2)$ with $\norm{f}_{L^\infty(L^2)}=O(e^{-\gamma_0 t})$, then for $\alpha_0=min(\alpha,\gamma_0);$ the solution decays exponentially with order $O(e^{-\alpha_0 t}).$
  
 \item If $f\in {L^\infty(L^2)}$, we obtain  regularity results proved in Lemmas \ref{2.1}-\ref{nl2} are valid uniformly in time for $\alpha =0.$
  \end{enumerate}
\end{remark}
% {\bf Remark.} 

 \subsection{Existence and Uniqueness of strong solution}
  Before, proving existence and uniqueness of a strong solution, we first prove the following 
  monotonicty property for our subsequent use.
\begin{lemma}\label{2.6}
For $u$ and $v$ $\in H^1_0$, there holds
 \begin{equation*}
  \left((1+\norm{\nabla u}^2)\nabla u-(1+\norm{\nabla v}^2)\nabla v,\nabla(u-v)\right)\geq \norm{\nabla(u-v)}^2.
 \end{equation*}
 \end{lemma}
 \begin{proof}
Note that
\begin{align*}
 \left((1+\norm{\nabla u}^2)\nabla u-(1+\norm{\nabla v}^2)\nabla v,\nabla(u-v)\right)=&\norm{\nabla(u-v)}^2+(\norm{\nabla u}^2\nabla u-\norm{\nabla v}^2\nabla v,\nabla(u-v))\\
  =&\norm{\nabla(u-v)}^2+\left(\norm{\nabla u}^2\nabla(u-v),\nabla(u-v)\right)\\+
  &\left((\norm{\nabla u}^2-\norm{\nabla v}^2)\nabla v,\nabla(u-v)\right)\\
  =&\norm{\nabla(u-v)}^2+\norm{\nabla u}^2\norm{\nabla(u-v)}^2\\+
  &(\norm{\nabla u}^2-\norm{\nabla v}^2)\left(\nabla v,\nabla(u-v)\right).
 \end{align*}
Now the term $(\norm{\nabla u}^2-\norm{\nabla v}^2)\left(\nabla v,\nabla(u-v)\right)$ can be written as 
 \begin{align*}
  &(\norm{\nabla u}^2-\norm{\nabla v}^2)\left(\nabla v,\nabla(u-v)\right)\\
  =&(\norm{\nabla u}^2-\norm{\nabla v}^2)\left(\nabla(u+v),\nabla(u-v)\right)-(\norm{\nabla u}^2-\norm{\nabla v}^2)(\nabla u,\nabla(u-v))\\
  \geq&(\norm{\nabla u}^2-\norm{\nabla v}^2)^2-(\norm{\nabla u}^2-\norm{\nabla v}^2)\norm{\nabla u}\norm{\nabla(u-v)}\\
  \geq&(\norm{\nabla u}^2-\norm{\nabla v}^2)^2-\frac{1}{2}(\norm{\nabla u}^2-\norm{\nabla v}^2)^2-\frac{1}{2}\norm{\nabla u}^2\norm{\nabla(u-v)}^2\\
  =&\frac{1}{2}(\norm{\nabla u}^2-\norm{\nabla v}^2)^2-\frac{1}{2}\norm{\nabla u}^2\norm{\nabla(u-v)}^2.
 \end{align*}
Therefore,
\begin{align*}
  \left((1+\norm{\nabla u}^2)\nabla u-(1+\norm{\nabla v}^2)\nabla v,\nabla(u-v)\right)\geq&\norm{\nabla(u-v)}^2+\frac{1}{2}(\norm{\nabla u}^2-\norm{\nabla v}^2)^2\\
  +&\frac{1}{2}\norm{\nabla u}^2\norm{\nabla(u-v)}^2\geq\norm{\nabla(u-v)}^2.
\end{align*}
This completes the rest of the proof.
 \end{proof}
 
 \begin{theorem}
  Suppose that $u_0\in H^1_0(\Omega)$ and $f\in L^\infty(L^2).$ Then for any finite $T>0$, the problem \eqref{eq1}-\eqref{eq3} admits a unique global strong  solution $u$ 
  for $t\in (0,T]$ satisfying $$u\in C([0,T],H^1_0)\cap L^2([0,T],H^2), u_t\in L^2([0,T],L^2).$$ 
 \end{theorem}
\begin{proof}
For a proof of existence, one can apply Bubnov-Galerkin method and compactness arguments of Lions in a standard way, see also \cite{chipot2}. 

For uniqueness, we prove it by contradiction. Assume contrary, then there exist two distinct solutions 
$u_1$ and $u_2$ of the problem \eqref{eq1} satisfying 
%$u_i , i=1,2 $ satisfy
 \begin{align*}
  (u_{it},v)+\left((1+\norm{\nabla u_i}^2)\nabla u_i,\nabla v\right)=(f,v), i=1,2.
 \end{align*}
 With $w=u_1-u_2$, $w$ now satisfies 
 \begin{align*}
  (w_t,v)+\left((1+\norm{\nabla u_1}^2)\nabla u_1-(1+\norm{\nabla u_2}^2)\nabla u_2,\nabla v\right)=0
 \end{align*}
 Substitute $v=u_1-u_2=w$ to obtain
 \begin{align*}
  (w_t,w)+\left((1+\norm{\nabla u_1}^2)\nabla u_1-(1+\norm{\nabla u_2}^2)\nabla u_2,\nabla (u_1-u_2)\right)=0
 \end{align*}
 Using monotonicity property given in Lemma \ref{2.6}, we observe that
 \begin{align*}
  \left((1+\norm{\nabla u_1}^2)\nabla u_1-(1+\norm{\nabla u_2}^2)\nabla u_2,\nabla (u_1-u_2)\right)\geq\norm{\nabla(u_1-u_2)}^2=\norm{\nabla w}^2\geq 0.
 \end{align*}
 Consequently,
 \begin{align*}
  \frac{d}{dt}\norm{w}^2\leq 0.
 \end{align*}
Since $w(0)=0$, it follows that $w=0$ which leads to a contradiction. Hence, the solution is unique. This completes the rest of the proof.
\end{proof}

As a consequence of Lemmas ~\ref{2.1} and  \ref{2.2}, we obtain the following results:
 From \eqref{rm1}, we note that the ball $B_{\rho_0}(0)$ \text{in} $L^2(\Omega)$ is absorbing in $L^2(\Omega)$ 
 with $\rho_0=\frac{1}{\sqrt{\alpha\lambda_1}}\norm{f}_{L^\infty(L^2)}.$ Specially, for any $R>0,$ there exists
 $t_0= t_0(R,\rho_0) > 0$ such that for $t\geq t_0=\frac{1}{2\alpha}log\left(\frac
 {2R^2-{\rho}^2_0}{{\rho}^2_0}\right)$
 $$B_{R}(0)\subset B_{\rho_0}(0).$$

To provide a quick sketch of its proof, note that for any $R>0$ with $u_0\in B_R(0)$, 
\begin{align*}
\norm{u(t)}\leq e^{-2\alpha t}R^2+\frac{1}{2}{\rho_0}^2(1-e^{-2\alpha t})=e^{-2\alpha t}(R^2-\frac{1}{2}{\rho_0}^2)+\frac{1}{2}{\rho_0}^2 \leq {\rho_0}^2,
\end{align*}
provided $e^{-2\alpha t}(R^2-\frac{1}{2}{\rho_0}^2)\leq \frac{1}{2}{\rho_0}^2$, that is, $e^{2\alpha t}\geq \frac{2R^2-{\rho_0}^2}{{\rho_0}^2}$. Now  taking $log$ both sides, it follows that, $t\geq \frac{1}{2\alpha}\;log\left(\frac{2R^2-{\rho_0}^2}{{\rho_0}^2}\right)=t_0.$
Hence, $B_{\rho_0}(0)$ is an absorbing ball in $L^2(\Omega).$
Similarly from \eqref{temp3} in Lemma \ref{2.2}, it follows that for any $R>0$ and $u_0\in B_{R}(0) \subset 
H^1_0,$ 
$$u(t)\leq \rho_1 \quad \text{for} \;\;t\geq t_1=\frac{1}{2\alpha}\;log\left(\frac{2R^2
-{\rho_1}^2}{{\rho_1}^2}\right),$$ 
where 
$\rho_1=\frac{1}{\sqrt{\alpha}}\norm{f}_{L^\infty(L^2)}.$ Therefore, $B_{\rho_1}(0)$ is an absorbing set in 
$H^1_0(\Omega)$ for the equation \eqref{temp1}.
%\end{remark}
Thus, we have the following result
\begin{theorem}
The problem \eqref{temp1}-\eqref{temp3} admits a global attractor in $L^2$  as well in $H^1_0.$
\end{theorem}
\section{Semidiscrete Galerkin Method}
This section deals with semidiscrete Galerkin approximation keeping time variable continuous and proves
optimal error estimates.

 Given a regular triangulation  $\mathcal{T}_h$ of $\overline{\Omega}$,
 let $h_K=\text{diam}(K)$ for all $K\in \mathcal{T}_h$ and $h=\displaystyle\max_{ K\in \mathcal{T}_h} h_K$\\
 Set
 $$V_h=\left\{v_h\in C^0(\overline{\Omega} ):\hspace{0.1cm} v_h\Big|_K \in \mathcal{P}_{1}(K) \quad \forall\quad K\in \mathcal{T}_{h}\hspace{0.1cm}\mbox{with}\hspace{0.1cm} v_h=0 \hspace{0.1cm}\mbox{on}\hspace{0.1cm}\partial \Omega\right\}.$$

Under an additional assumption that the family of  triangulation $\mathcal{T}_h$  is quasi-uniform, the following
inverse inequality holds
$$\norm{\nabla \chi}\leq Ch^{-1}\norm{\chi} \quad \forall \chi\in V_h.$$
 Now the semidiscrete approximation $u_h(t)$ of\eqref{temp1} is to find $u_h(t)\in V_h$ for $t>0$ such that 
\begin{equation}\label{eq4}
  (u_{ht},\chi)+\left(1+\norm{\nabla u_h}^2_{L^2(\Omega)}\right)(\nabla u_h,\nabla \chi)=(f,\chi) \qquad \forall \chi \in V_h
\end{equation}
with $u_h(0)=u_{0h}\in V_h$ to be defined later.
\begin{theorem}
 For any $u_{0h}\in V_h$, there exists a unique solution $u_h\in C^1([0,\infty];V_h)$ satisfying \eqref{eq4}.
\end{theorem}
\begin{proof}
Since $V_h$ is finite dimensional, \eqref{eq4} leads to a 
a system of nonlinear ODE's. An appeal to the Picard's theorem yields the existence of a unique solution $u_h(t)$
locally, that is, there exists $t=t^*>0$ such that \eqref{eq4} has a unique solution $u_h(t)$ for $t\in(0,t^{*})$.\\
For global existence, we use continuation argument provided $\norm{u_h(t)}$ is bounded for all $t>0$. Now
choose $\chi=u_h$ in \eqref{eq4} to obtain as in Lemma \ref{temp1} with $0<\alpha <\frac{\lambda_1}{2}$
\begin{align}
 \norm{u_h(t)}^2+\min(\beta,2)e^{-2\alpha t}&\int_{0}^{t}(\norm{\nabla u_h(s)}^2+\norm{\nabla u_h(s)}^4)ds\notag\\
 &\leq e^{-2\alpha t}\norm{u_{0h}}^2+\frac{1}{2\alpha\lambda_1}\norm{f}^2_{L^\infty(L^2)}(1-e^{-2\alpha t}),\label{eq3.2}
\end{align}
 where $\beta=(1-\dfrac{2\alpha}{\lambda_1})>0$. Note $u_{0h}$ an approximation of $u_0$ in $V_h$ and $\norm{u_{0h}}$ can be made bounded by $\norm{u_0}$. The
result on global existence of a unique solution $u_h(t)$ of \eqref{eq4} now follows for all $t>0$. This completes
the rest of the proof.
\end{proof}
As a consequence, the following result holds as in the continuous case.
\begin{proposition}
There exists a bounded absorbing set $B_{R}(0)$ in $V_h$ for \eqref{eq4}, that is, for $R>0$ and $u_{0h}
\in B_R(0),$ there exists $t=t_0(\norm{u_{0h}})$ such that for $t\geq t_0$, $u_h(t)\in B_{\rho_{0}}(0)$, where 
$\rho_0=\dfrac{1}{\sqrt{\alpha\lambda_1}}\norm{f}_{L^\infty(L^2)}.$
\end{proposition}
Thus, as in continuous case, the following theorem holds.
\begin{theorem}
 The equation \eqref{eq4} has a global attractor $\mathcal{A}_h$ which $attracts$ bounded set in $V_h.$
\end{theorem}

\subsection{A priori bounds}
 We now introduce {\it{discrete Laplacian}} $\Delta_h:V_h\longrightarrow V_h$ by 
  \begin{align}\label{temp6}
  (-\Delta_h v_h,w_h)=(\nabla v_h,\nabla w_h) \qquad \forall v_h, w_h\in V_h.
  \end{align}
\begin{lemma}\label{lemma3.1}
 Let $u_{0h}$ is an approximation of $u_0$ and $f\in L^\infty(L^2)$. Then for $0<\alpha<\frac{\lambda_1}{2}$, there holds:
 \begin{align*}
  \norm{\nabla u_h(t)}^2+\beta e^{-2\alpha t}\int_{0}^{t} e^{2\alpha s}\norm{\Delta_h u_h(s)}^2 ds
  & +2e^{-2\alpha t}\int_{0}^{t}e^{2\alpha s}\norm{\nabla u_h(s)}^2\norm{\Delta_hu_h(s)}^2ds\\
  &\leq e^{-2\alpha t}\norm{\nabla u_{0h}}^2+\norm{f}^2(1-e^{-2\alpha t})=\hat K_{4}(t)\\
  &\leq\norm{\nabla u_0}^2+\frac{1}{2\alpha}\norm{f}^2_{L^\infty(L^2)}=K_{1}.
 \end{align*}
\end{lemma}
\begin{proof}
 Proof is similar to the proof of  Lemma~\ref{2.2}.
\end{proof}

\begin{lemma}\label{lemma3.2}
Let $u_{0h}$ is an approximation of $u_0$, $u_0\in H^1_0(\Omega)$ and $f\in L^\infty(L^2)$. Then there holds:
  \begin{align*}
   2e^{-2\alpha t}\int_{0}^{t}&e^{2\alpha s}\norm{u_{hs}}^2ds+(2+\norm{\nabla u_h(t)}^2)\norm{\nabla u_h(t)}^2\\
  &\leq(2+\norm{\nabla u_{h0}}^2)\norm{\nabla u_{0h}}^2e^{-2\alpha t}+\frac{1}{\alpha}\norm{f}^2_{L^\infty(L^2)}(1-e^{-2\alpha t})
  +\alpha(1+\dfrac{4}{\beta})\widehat K_0(t)=\widehat K_5(t)\\
  &\leq(2+\norm{\nabla u_{0h}}^2)\norm{\nabla u_{0h}}^2+\frac{1}{\alpha}\norm{f}^2_{L^\infty(L^2)}+\alpha(1+\dfrac{4}{\beta})K_0=K_5.
  \end{align*}
 \end{lemma}
\begin{proof}
 Proof is similar to the proof of the  Lemma \ref{2.3}.
\end{proof}
\begin{lemma}\label{lemma3.3}
%    Let $u_0\in H^1_0(\Omega)$ , $f\in L^\infty(L^2)$, and $f_t\in L^2(H^{-1})$. Then, there holds
Let $u_{0h}$ is an approximation of $u_0$, $f\in L^\infty(L^2)$ and $f_t\in L^2(H^{-1})$. Then there holds:
 \begin{align*}
 &\tau(t)\norm{u_{ht}(t)}^2+e^{-2\alpha t}\int_{0}^{t}\tau(s)e^{2\alpha s}(\beta+2\norm{\nabla u_h}^2)\norm{\nabla u_{hs}}^2ds+4e^{-2\alpha t}\int_{0}^{t}\tau(s)e^{2\alpha s}(\nabla u_{hs},\nabla u_h)^2ds\\
   &\leq e^{-2\alpha t}\int_{0}^{t}e^{2\alpha s}\norm{f_s}^2_{-1}ds+K_5,
 \end{align*}
 where $\tau(t)=min(t,1).$
 \end{lemma}
 \begin{proof}
  Differentiating of \eqref{eq4} with respect to time yields
  \begin{align}
   (u_{htt},v)+\left((1+\norm{\nabla u_h}^2)\nabla u_{ht},\nabla v\right)+\left(2(\nabla u_{ht},\nabla u_h)\nabla u_h,\nabla v\right)=(f_t,v).\label{eq10}
  \end{align}
 Substitute $v=u_{ht}$ in \eqref{eq10} to obtain
 \begin{align}
  \frac{1}{2}\frac{d}{dt}\norm{u_{ht}}^2+(1+\norm{\nabla u_h}^2)\norm{\nabla u_{ht}}^2+2(\nabla u_{ht},\nabla u_h)^2=(f_t,u_{ht}).\label{eq11}
 \end{align}
 Set $\tau(t)=min(t,1)$ and 
 multiplying \eqref{eq11} by $2\tau e^{2\alpha t}$,$\alpha>0$ and using Poincare's inequality it follows that
 \begin{align*}
  \frac{d}{dt}(\tau e^{2\alpha t}\norm{u_{ht}}^2)+\tau(t)e^{2\alpha t}\left(1-\frac{2\alpha}{\lambda_1}\right)&\norm{\nabla u_{ht}}^2+\\
  &2\tau(t)e^{2\alpha t}(1+\norm{\nabla u_h}^2)\norm{\nabla u_{ht}}^2+4\tau(t)e^{2\alpha t}(\nabla u_{ht},\nabla u_h)^2\\
  &\leq e^{2\alpha t}\norm{f_t}^2_{-1}+e^{2\alpha t}\norm{u_{ht}}^2.
 \end{align*}
 Integrating with respect to time from $0$ to $t$ and multiplying the resulting inequality by $e^{-2\alpha t}$ to obtain
 \begin{align}
  &\tau(t)\norm{u_{ht}(t)}^2+e^{-2\alpha t}\int_{0}^{t}\tau(s)e^{2\alpha s}(\beta+2\norm{\nabla u_h}^2)\norm{\nabla u_{hs}}^2ds+4e^{-2\alpha t}\int_{0}^{t}\tau(s)e^{2\alpha s}(\nabla u_{hs},\nabla u_h)^2ds\notag\\
   &\leq e^{-2\alpha t}\int_{0}^{t}e^{2\alpha s}\norm{f_s}^2_{-1}ds+e^{-2\alpha t}\int_{0}^{t}e^{2\alpha s}\norm{u_{hs}}^2\;ds\label{eq12}.
%   &\leq e^{-2\alpha t}\left((1+\norm{\nabla u_0}^2)\norm{u_0}_{H^2}+\norm{f_0}\right)+e^{-2\alpha t}\int_{0}^{t}e^{2\alpha s}\norm{f_s}^2_{-1}ds.\notag
 \end{align}
 Now using Lemma \ref{lemma3.2}  the second term in the right hand side is bounded and hence, this completes the rest of the proof.
 \end{proof}
 \begin{lemma}\label{lemma3.4}
  Let $u_{0h}$ is an approximation of $u_0$, $u_0\in H^1_0(\Omega)$, $f\in L^\infty(L^2)$, and $f_t\in L^2(H^{-1})$. Then, there holds:
  \begin{align*}
   \tau(t)(1+\norm{\nabla u_h(t)}^2)\norm{\Delta u_h(t)}^2&\leq 2\tau\norm{f(t)}^2_{L^\infty(L^2)}+2\tau\norm{u_{ht}(t)}^2_{L^\infty(L^2)}\\
   &\leq 2\norm{f(t)}^2_{L^\infty(L^2)}+2e^{-2\alpha t}\int_{0}^{t}e^{2\alpha s}\norm{f_s}^2_{-1}ds+K_5.
  \end{align*}
 \end{lemma}
\begin{proof}
 Substitute $v=-\Delta_hu_h$ in the weak formulation \eqref{eq4} to obtain 
 \begin{equation*}
  (1+\norm{\nabla u_h}^2)\norm{\Delta_h u_h}^2=-(f,\Delta_hu_h)+(u_t,\Delta_hu_h),
 \end{equation*}
and using Young's inequality, and \eqref{temp6},we now bound
 \begin{align*}
  (1+\norm{\nabla u_h}^2)\norm{\Delta_hu_h}^2&\leq\norm{f}^2_{L^\infty(L^2)}+\norm{u_{ht}}^2_{L^\infty(L^2)}+\frac{1}{2}\norm{\Delta_hu_h}^2\\
  &\leq\norm{f}^2_{L^\infty(L^2)}+\norm{u_{ht}}^2_{L^\infty(L^2)}+\frac{1}{2}(1+\norm{\nabla u_h}^2)\norm{\Delta_hu_h}^2.                                        
 \end{align*}
 Therefore, multiplying by the resulting inequality by $\tau$ it follows that
 \begin{align}
  \tau(1+\norm{\nabla u_h}^2)\norm{\Delta_hu_h}^2\leq2\tau\norm{f}^2_{L^\infty(L^2)}+2\tau\norm{u_{ht}}^2_{L^\infty(L^2)}.
 \end{align}
 From Lemma \ref{lemma3.3} applying the bound of $\tau\norm{u_t}^2$, we obtain bound for $\norm{\Delta_hu_h}.$
 This completes the rest of the proof.
\end{proof}

\subsection{A priori Error estimates}
This subsection focuses on  error estimates of the semidiscrete Galerkin approximation.

Let $\widetilde u_h(t)\in V_h$ be the Ritz-projection of $u(t)\in H^1_0(\Omega)$ defined by
\begin{equation}\label{eq5}
 \left(\nabla(u-\widetilde u_h),\nabla \chi\right)=0\hspace{0.1cm} \forall\chi\in V_h.
\end{equation}
For each $t>0$, $\widetilde u_h(t)\in V_h$ is welldefined for a given $u(t)$. With $\eta =u-\widetilde u_h,$ the following error estimates hold:
\begin{align}\label{eq3.4}
\norm{\eta}_j\leq C h^{\min(2,m)-j}\;\norm{u}_m, \text{and} \norm{\eta _t}\leq C h^{\min(2,m)-j}\norm{u_t}_m,\;\;
j=0,1 \;\mbox{ and }\; m=1,2
\end{align}
For a proof, refer to Thomee \cite{thomee}.
 Now split
$u-u_h=(u-\widetilde u_h)-(u_h-\widetilde u_h):=\eta-\theta$\\
Since estimates of $\eta$ are known, it is enough to estimate $\theta$. Using \eqref{temp1}, \eqref{eq4} and 
\eqref{eq5}, we arrive at an equation in $\theta$ as 
\begin{align}
 (\theta_t,\chi)+\Big((1+\norm{\nabla u_h}^2)\nabla u_h-(1+\norm{\nabla\widetilde u_h}^2)&\nabla\widetilde u_h,\nabla\chi\Big)=(\eta _t,\chi)\notag\\
 &+\Big((1+\norm{\nabla u}^2)\nabla u-(1+\norm{\nabla\widetilde u_h}^2)
 \nabla\widetilde u_h,\nabla \chi\Big)\label{eq3.5}
\end{align}
 \begin{theorem}\label{3.1}
Let $u_0\in H^1_0(\Omega)$, $f\in L^\infty(L^2)$ and $f_t \in L^{2}(H^{-1}).$ Then there holds:
\begin{align*}
 &\norm{\theta(t)}^2+\beta e^{-2\alpha t}\int_{0}^{t}e^{2\alpha s}\norm{\nabla\theta(s)}^2ds\\
 &\leq C(K_1)h^2e^{-2\alpha t}\Big(\norm{\nabla u_0}^2+\int_{0}^{t}e^{2\alpha s}\norm{f_s}^2_{-1}ds+\left((1+\norm{\nabla u_0}^2)^2\norm{u_0}^2_{H^2}+\norm{f_0}^2\right)\Big),
%  &\tau\norm{\theta(t)}^2+\tau\beta e^{-2\alpha t}\int_{0}^{t}e^{2\alpha s}\norm{\nabla\theta(s)}^2ds\\
%  &\leq C(K_1,K_3)h^2e^{-2\alpha t}\left(\norm{\nabla u_0}^2+\int_{0}^{t}e^{2\alpha s}\norm{f_s}^2_{-1}ds\right)
\end{align*}
 where $K_1$ depends on $\norm{\nabla u_0}$ and $\norm{f}_{L^\infty(L^2)}$.
 \end{theorem}
\begin{proof}
 Set $\chi=\theta$ in \eqref{eq3.5} and use also the monotonicity property to obtain 
\begin{align}
 \frac{1}{2}\frac{d}{dt}\norm{\theta (t)}^2+\norm{\nabla \theta}^2 &\leq\norm{\eta_t}\norm{\theta}+
 \left(\norm{\nabla u}^2-\norm{\nabla\widetilde u_h}^2\right)\norm{\nabla u}\norm{\nabla\theta}\notag\\
 &\leq\left(\frac{1}{\sqrt(\lambda_1)}\norm{\eta_t}+(\norm{\nabla u}+\norm{\nabla\widetilde u_h})\norm{\nabla \eta}\norm{\nabla u}\right)\norm{\nabla \theta}\label{eqn3.7}\\
 &\leq\frac{1}{\lambda_1}\norm{\eta_t}^2+\left(\norm{\nabla u}+\norm{\nabla\widetilde u_h}\right)^2
 \norm{\nabla \eta}^2\norm{\nabla u}^2+\frac{1}{2}\norm{\nabla \theta}^2.\notag
 \end{align} 
 Here, we have used Poincare inequality and Youngs inequality.
 Multiply by $2e^{2\alpha t}$, $\alpha >0$ and rewrite it as 
 \begin{align*}
  \frac{d}{dt}(e^{2\alpha t}\norm{\theta (t)}^2)-&2\alpha e^{2\alpha t}\norm{\theta}^2+e^{2\alpha t}\norm{\nabla \theta}^2\\
  &\leq\frac{1}{\lambda_1}e^{2\alpha t}\norm{\eta_t}^2+2\left(\norm{\nabla u}^2+\norm{\nabla\widetilde u_h}^2\right)\norm{\nabla u}^2(e^{2\alpha t}\norm{\nabla \eta}^2).  
 \end{align*}
Using Poincare's inequality $\norm{\theta}^2\leq\frac{1}{\lambda_1}\norm{\nabla \theta}^2$ with $\alpha>0$ such that
$\beta=(1-\frac{2\alpha}{\lambda_1})>0$ and integrating with respect to $t$ from $0$ to $t$ to obtain
\begin{align}
 \norm{\theta(t)}^2+\beta e^{-2\alpha t}\int_{0}^{t}e^{2\alpha s}\norm{\nabla \theta(s)}^2ds\leq e^{-2\alpha t}\norm{\theta(0)}^2+
 \frac{1}{\lambda_1}e^{-2\alpha t}\int_{0}^{t}e^{2\alpha s}\norm{\eta_s}^2ds\notag\\+2e^{-2\alpha t}\int_{0}^{t}\left(\norm{\nabla u}^2+
 \norm{\nabla\widetilde u_h}^2\right)\norm{\nabla u}^2e^{2\alpha s}\norm{\nabla \eta(s)}^2ds.\label{3.8}
\end{align}
 With a choice $u_{h0}=\widetilde u_h(0)$,$\theta(0)=0$. But with $u_{h0}=P_hu_0 \hspace{0.1cm} \text{or } u_{h0}=I_hu_0$, where $P_h$ and $I_h$, respectively, are
 $L^2$- projection and interpolant onto $V_h$, then
 $$\norm{\theta(0)}\leq\norm{u_{h0}-u_0}+\norm{u_0-\widetilde u_h(0)}\leq Ch\norm{\nabla u_0}.$$
 Then using regularity result in Lemma \ref{2.2} we arrive at 
 \begin{align*}
 \norm{\theta(t)}^2+&\beta e^{-2\alpha t}\int_{0}^{t}e^{2\alpha s}\norm{\nabla \theta(s)}^2ds\\ &\leq C(K_1)h^2e^{-2\alpha t}
 \left(\norm{\nabla u_0}^2+\int_{0}^{t}e^{2\alpha s}\norm{\nabla u_s}^2+\int_{0}^{t}e^{2\alpha s}\norm{\Delta u(s)}^2\right).
 \end{align*}
An application of Lemma \ref{2.2}, \ref{2.4} yields the final result. This completes the rest of the proof.
\end{proof}

Since, the estimate $\norm{\nabla \eta} \leq Ch\norm{u}_2$ is known, it is enough to
prove the estimate of
$\nabla\theta$. 
\begin{theorem}\label{3.2}
Let $u_0\in H^2\cap H^1_0(\Omega)$, $f\in L^\infty(L^2)$ and $f_t\in L^{2}(H^{-1})$. Then, there holds:
 \begin{align*}
 e^{-2\alpha t}\int_{0}^{t}\tau(s)e^{2\alpha s}\norm{\theta_s}^2\;ds&+\tau(t)(1+\norm{\nabla u_h}^2)\norm{\nabla\theta(t)}^2\\
 &\leq C(K_1,K_5)h^2e^{-2\alpha t}\Big((1+\norm{\nabla u_0)}^2)^2\norm{u_0}^2_{H^2}+\norm{f_0}^2+\int_{0}^{t}e^{2\alpha s}\norm{f_s}^2_{-1} ds\Big).
 \end{align*}
\end{theorem}
\begin{proof}
 Setting $\chi=\theta_t$ in \eqref{eq3.5}, it follows that
 \begin{align}
  (\theta_t,\theta_t)+&\left((1+\norm{\nabla u_h}^2)\nabla u_h-(1+\norm{\nabla\widetilde u_h}^2)\nabla\widetilde u_h,\nabla\theta _t\right)=(\eta_t,\theta_t)\notag\\
  &\left((1+\norm{\nabla u}^2)\nabla u-(1+\norm{\nabla\widetilde u_h}^2)\nabla\widetilde u_h,\nabla\theta _t\right).\label{eq4.2}
 \end{align} 
Now multiplying by $e^{2\alpha t}$ , $\alpha>0$ and applying Ritz projection the equation \eqref{eq4.2} can be written as
\begin{align*}
     e^{2\alpha t} \norm{\theta _t}^2+ e^{2\alpha t}(\nabla \theta,\nabla \theta _t)+& e^{2\alpha t}\left(\norm{\nabla u_h}^2\nabla u_h-\norm{\nabla\widetilde u_h}^2\nabla\widetilde u_h),\nabla\theta_t\right)\\
     = e^{2\alpha t}(\eta_t,\theta _t)+& e^{2\alpha t}\left(\norm{\nabla u}^2\nabla u-\norm{\nabla\widetilde u_h}^2\nabla\widetilde u_h,\nabla\theta _t\right)
\end{align*}
 Now rewrite it as 
 \begin{align}
  e^{2\alpha t}\norm{\theta_t}^2+&\frac{1}{2}\frac{d}{dt}(e^{2\alpha t}\norm{\nabla \theta}^2)-\alpha e^{2\alpha t}\norm{\nabla \theta}^2+e^{2\alpha t}
  \left(\norm{\nabla u_h}^2\nabla u_h-\norm{\nabla\widetilde u_h}^2\nabla\widetilde u_h,\nabla \theta_t\right)\notag\\&=e^{2\alpha t}(\eta_t,\theta_t)+
  e^{2\alpha t}\left(\norm{\nabla u}^2\nabla u-\norm{\nabla\widetilde u_h}^2\nabla\widetilde u_h,\nabla \theta_t\right).\label{eq4.3}
 \end{align}
 A use of the Ritz projection shows
 \begin{align}
  e^{2\alpha t}\left(\norm{\nabla u}^2\nabla u-\norm{\nabla\widetilde u_h}^2\nabla\widetilde u_h,\nabla \theta_t\right)&=e^{2\alpha t}\left((\norm{\nabla u}^2-\norm{\nabla\widetilde u_h}^2)\nabla\widetilde u_h,\nabla\theta_t\right)\notag\\
  &=e^{2\alpha t}\left((\norm{\nabla u}^2-\norm{\nabla\widetilde u_h}^2)\nabla u,\nabla\theta_t\right)\notag\\
  &-e^{2\alpha t}\left((\norm{\nabla u}^2-\norm{\nabla\widetilde u_h})\nabla{(u-\widetilde u_h),\nabla\theta_t}\right)\notag\\
  &=-e^{2\alpha t}\left((\norm{\nabla u}^2-\norm{\nabla\widetilde u_h}^2)\Delta u,\theta_t\right)\notag\\
  &\leq e^{2\alpha t}(\norm{\nabla u}+\norm{\nabla\widetilde u_h})\norm{\nabla\eta}\norm{\Delta u}\norm{\theta_t}.\label{eq4.4}  
 \end{align}
 For the third term on the left-hand side of \eqref{eq4.3}, rewrite it as
 \begin{align}
 e^{2\alpha t}\left(\norm{\nabla u_h}^2\nabla u_h-\norm{\nabla\widetilde u_h}^2\nabla\widetilde u_h,\nabla \theta_t\right)&=\frac{1}{2}\frac{d}{dt}(\norm{\nabla\theta}^2)e^{2\alpha t}\norm{\nabla u_h}^2\notag\\
  &-e^{2\alpha t}\left((\norm{\nabla u_h}^2-\norm{\nabla\widetilde u_h}^2)\nabla(u-\widetilde u_h),\nabla\theta_t\right)\notag\\
  &+e^{2\alpha t}\left((\norm{\nabla u_h}^2-\norm{\nabla\widetilde u_h}^2)\nabla u,\nabla\theta_t\right)\notag\\
  &=\frac{1}{2}\frac{d}{dt}(\norm{\nabla\theta}^2)e^{2\alpha t}\norm{\nabla u_h}^2\notag\\
  &-e^{2\alpha t}\left((\norm{\nabla u_h}^2-\norm{\nabla\widetilde u_h}^2)\Delta u,\theta_t\right).\label{eq4.5}
 \end{align}
  Similarly,
 \begin{align}
  &\frac{1}{2}\frac{d}{dt}(e^{2\alpha t}\norm{\nabla\theta}^2)+\frac{1}{2}\frac{d}{dt}(\norm{\nabla\theta}^2)e^{2\alpha t}\norm{\nabla u_h}^2\notag\\
  &=\frac{1}{2}\frac{d}{dt}\left((1+\norm{\nabla u_h}^2)e^{2\alpha t}\norm{\nabla\theta}^2\right)-e^{2\alpha t}\norm{\nabla\theta}^2(u_{ht},-\Delta_h u_h)
  -\alpha e^{2\alpha t}\norm{\nabla u_h}^2\norm{\nabla\theta}^2.\label{eq4.6}
 \end{align}
 Substitute \eqref{eq4.4},\eqref{eq4.5},\eqref{eq4.6} in \eqref{eq4.3} to obtain 
 \begin{align}
   e^{2\alpha t}\norm{\theta_t}^2+\frac{1}{2}\frac{d}{dt}\left((1+\norm{\nabla u_h}^2)e^{2\alpha t}\norm{\nabla\theta}^2\right)&\leq e^{2\alpha t}\norm{\eta_t}\norm{\theta_t}\notag\\
   &+e^{2\alpha t}\left((\norm{\nabla u_h}^2-\norm{\nabla\widetilde u_h}^2)\Delta u,\theta_t\right)\notag\\
   &+e^{2\alpha t}(\norm{\nabla u}+\norm{\nabla\widetilde u_h})\norm{\nabla\eta}\norm{\Delta u}\norm{\theta_t}\label{eq4.7}\\
   &+e^{2\alpha t}\left(\alpha(1+\norm{\nabla u_h}^2)+(u_{ht},-\Delta_h u_h)\right)\norm{\nabla\theta}^2.\notag
   \end{align}
Now multiplying the above inequality by $\tau$ with Young's inequality yields
\begin{align}
 \tau(t)e^{2\alpha t}\norm{\theta_t}^2+\frac{d}{dt}&\left(\tau(t)(1+\norm{\nabla u_h}^2)e^{2\alpha t}\norm{\nabla\theta}^2\right)\\
 &\leq6\tau(t)e^{2\alpha t}\Big(\frac{1}{2}\norm{\eta_t}^2+(\norm{\nabla u_h}^2+\norm{\nabla\widetilde u_h}^2)\norm{\nabla\theta}^2\norm{\Delta u}^2\label{eq4.8}\\
   &+(\norm{\nabla u}^2+\norm{\nabla\widetilde u_h}^2)\norm{\nabla\eta}^2\norm{\Delta u}^2\Big)\notag\\
   &+2e^{2\alpha t}\left((\tau(t)\alpha+\frac{1}{2})(1+\norm{\nabla u_h}^2)+\tau(t)\norm{u_{ht}}\norm{\Delta_h u_h}\right)\norm{\nabla\theta}^2.\notag
\end{align}
 An integration of  \eqref{eq4.8} with respect to time from $0$ to $t$ shows using $\norm{\nabla\widetilde u_h}\leq\norm{\nabla u}$(from Ritz projection) and
 multiplying the resulting inequality by $e^{-2\alpha t}$  that
\begin{align}
 e^{-2\alpha t}\int_{0}^{t}\tau(s)e^{2\alpha s}\norm{\theta_s}^2 ds&+\tau(t)(1+\norm{\nabla u_h}^2)\norm{\nabla\theta(t)}^2\notag\\
  &\leq Ce^{-2\alpha t}\int_{0}^{t}e^{2\alpha s}\left(\norm{\eta_s}^2+\norm{\nabla u}^2)\norm{\Delta u}^2\norm{\nabla\eta}^2\right)ds\label{eq4.9}\\
  &+Ce^{-2\alpha t}\int_{0}^{t}e^{2\alpha s}\Big((\norm{\nabla u_h}^2+\norm{\nabla u}^2)\norm{\Delta u}^2\notag\\
  &+(1+\alpha)(1+\norm{\nabla u_h}^2)+({\tau}^\frac{1}{2}\norm{u_{hs}})({\tau}^\frac{1}{2}\norm{\Delta_h u_h})\Big)\norm{\nabla\theta(s)}^2\;ds.\notag
\end{align}
 Consequently,
\begin{align}
e^{-2\alpha t}\int_{0}^{t}\tau(s)e^{2\alpha s}\norm{\theta_s}^2ds&+\tau(t)(1+\norm{\nabla u_h}^2)\norm{\nabla\theta(t)}^2\notag\\
&\leq Ch^2e^{-2\alpha t}\int_{0}^{t}e^{2\alpha s}\left(\norm{\nabla u_s}^2+\norm{\nabla u}^2)\norm{\Delta u}^2(\norm{\Delta u}^2)\right)ds\notag\\ 
&+Ce^{-2\alpha t}\int_{0}^{t}e^{2\alpha s}\Big((\norm{\nabla u_h}^2+\norm{\nabla u}^2)\norm{\Delta u}^2\label{eq4.10}\\
&+(1+\alpha)(1+\norm{\nabla u_h}^2)+({\tau}^\frac{1}{2}\norm{u_{hs}})({\tau}^\frac{1}{2}\norm{\Delta_h u_h})\Big)\norm{\nabla\theta(s)}^2\;ds.\notag
\end{align}
From Lemmas \ref{2.2}, \ref{2.4}-\ref{2.5}, the first term on right hand side is bounded by $$C(K_1)h^2e^{-2\alpha t}\left((1+\norm{\nabla u_0)}^2)^2\norm{u_0}^2_{H^2}+\norm{f_0}^2+\int_{0}^{t}e^{2\alpha s}\norm{f_s}^2_{-1}ds\right).$$ 
For bounding the second term on right hand side, apply  the previous theorem~\ref{3.1} to obtain a bound as
%is applied provided the term 
$$\left((\norm{\nabla u_h}^2+\norm{\nabla u}^2)\norm{\Delta u}^2+(1+\alpha)(1+\norm{\nabla u_h}^2)+({\tau}^\frac{1}{2}\norm{u_{hs}})({\tau}^\frac{1}{2}\norm{\Delta_h u_h})\right).$$ 
By applying Lemma  \ref{2.2}, \ref{2.5}, \ref{lemma3.1},  
 \ref{lemma3.3}, \ref{lemma3.4}, the above term is bounded. Consequently, the third term 
 in the right hand side is bounded by 
 $$C(K_1,K_5)h^2e^{-2\alpha t}\left((1+\norm{\nabla u_0)}^2)^2\norm{u_0}^2_{H^2}+\norm{f_0}^2+\int_{0}^{t}e^{2\alpha s}\norm{f_s}^2_{-1}ds\right).
 $$ 
Altogether, we arrive at
\begin{align}
 e^{-2\alpha t}\int_{0}^{t}\tau(s)e^{2\alpha s}\norm{\theta_s}^2\;ds&+\tau(t)(1+\norm{\nabla u_h}^2)
 \norm{\nabla\theta(t)}^2 \nonumber\\
 &\leq C(K_1,K_5)h^2e^{-2\alpha t}\Big((1+\norm{\nabla u_0}^2)^2\norm{u_0}^2_{H^2}+\norm{f_0}^2+\int_{0}^{t}e^{2\alpha s}\norm{f_s}^2_{-1} \;ds\Big).
\end{align}
This completes the rest of the proof.
\end{proof}

An application of  triangle inequality with the estimate of  $\norm{\nabla \eta}$ from \eqref{eq3.4} and the
estimate $\norm{\nabla \theta}$ from Theorem ~\ref{3.2} yields the following main result of this section.
\begin{theorem}\label{semi-discrete-main}
Let $u_0\in H^2\cap H^1_0(\Omega)$, $f\in L^\infty(L^2)$ and $f_t\in L^{2}(H^{-1})$. Then, there holds:
 \begin{align} 
 \norm{(u-u_h)(t)}^2+&\tau(t)\;\norm{\nabla (u-u_h)(t)}^2\notag\\
 &\leq C(K_1,K_5)h^2 \;e^{-2\alpha t}\Big((1+\norm{\nabla u_0}^2)^2\norm{u_0}^2_{H^2}
 +\norm{f_0}^2+\int_{0}^{t}e^{2\alpha s}\norm{f_s}^2_{-1} \;ds\Big)\label{eq:main-semidiscrete}.
 \end{align}
\end{theorem}

\begin{remark}
\begin{itemize}
\item [(i)] Note that from the theorem~\ref{semi-discrete-main}, the estimates are valid uniformly in time.
\item [(ii)] When $f=0$, or $f, f_t=O(e^{-\gamma_0 t}),$ the following exponential decay property  
for the error estimates holds:
\begin{align}
\norm{(u-u_h)(t)}+ \tau(t)\norm{\nabla (u-u_h)(t)} \leq C(K_1)h \;e^{-2\gamma t},
\end{align}
where $K_1$ depends on $\norm{\nabla u_0}^2,$ and $\gamma =\alpha,$ in case $f=0$ and $\gamma 
=\min(\alpha, \gamma_0)$ for $f=O(e^{-\gamma_0 t}).$ 
\end{itemize}
\end{remark}

\section{Backward Euler Method}
This section is devoted to a completely discrete scheme which is based on a backward Euler method. Let $\{{t_n}\}_{n=0}^{N}$ be a uniform partition of [0,T], and $t_n=nk,$ with time step $k>0.$ For smooth function $\phi$ defined on $[0,T]$,
 set $\phi ^n=\phi(t_n)$ and $\bar{\partial}_t\phi^n=\frac{(\phi^n-\phi^{n-1})}{k}$.\\
 Now the backward Euler method applied to \eqref{eq4} determines a sequence of functions $\{{U^n}\}_{n\geq 1}\in V_h$ as solution of
 \begin{align}
 &(\bar{\partial}_tU^n,\varphi_h)+(1+\norm{\nabla U^n}^2)(\nabla U^n,\nabla \varphi_h)=(f^n,\varphi_h) \quad \forall \varphi_h \in V_h,\label{eq5.1}\\
  &U^0=u_{0h}.\notag
 \end{align}
 Now we derive {\it{a priori}} bounds for the solution$\{U^n\}_{n\geq 1}$.
\begin{lemma}\label{4.1}
  The discrete solution $U^N,$ $N\geq 1$ of \eqref{eq5.1} satisfies 
\begin{equation}
   \norm{U^N}\leq \norm{U^0}+2k\sum_{n=1}^{N}\norm{f^n}.
\end{equation}
\end{lemma}
\begin{proof}
 Set $\varphi_h=U^n$ in \eqref{eq5.1} and obtain
 \begin{equation}
  (\bar{\partial}_tU^n,U^n)+(1+\norm{\nabla U^n}^2)\norm{\nabla U^n}^2=(f^n,U^n).
 \end{equation}
Note that
\begin{align*}
 (\bar{\partial}_tU^n,U^n)=\frac{1}{2}\bar{\partial}_t\norm{U^n}^2+\frac{k}{2}\norm{\bar{\partial}_t U^n}^2.
\end{align*}
Therefore, 
\begin{align*}
 \frac{1}{2}\bar{\partial}_t\norm{U^n}^2+\frac{k}{2}\norm{\bar{\partial}_t U^n}^2+(1+\norm{\nabla U^n}^2)\norm{\nabla U^n}^2&=(f^n,U^n)\\
 &\leq \norm{f^n}\norm{U^n}.
\end{align*}
Consequently,
\begin{equation*}
 \frac{1}{2}\bar{\partial}_t\norm{U^n}^2\leq \norm{f^n}\norm{U^n}.
\end{equation*}
That is,
\begin{equation*}
 (\norm{U^n}^2-\norm{U^{n-1}})\leq 2k\norm{f^n}\norm{U^n}.
\end{equation*}
Sum it up from $n=1$ to $N$ to obtain
 \begin{equation}\label{eq5.3}
  \norm{U^N}^2\leq \norm{U^0}^2+2k\sum_{n=1}^{N}\norm{f^n}\norm{U^n}.
 \end{equation}
Let $N^*\in\{{0,\cdots,N}\}$ such that $U^{N^*}=\max\limits_{0\leq n\leq N}\norm{U^n}$.\\
Since $\eqref{eq5.3}$ is true for $N=N^*$, therefore, 
\begin{align*}
 \norm{U^{N^*}}^2&\leq \norm{U^0}^2+2k\sum_{n=1}^{N}\norm{f^n}\norm{U^n}\\
 &\leq \left(\norm{U^0}+2k\sum_{n=1}^{N}\norm{f^n}\right)\norm{U^{N^*}}.
\end{align*}
Consequently,
\begin{align*}
 \norm{U^N}\leq\norm{U^{N^*}}\leq\left(\norm{U^0}+2k\sum_{n=1}^{N}\norm{f^n}\right).
\end{align*}
This completes the rest of the proof.
\end{proof}
\begin{remark}\label{rm4.1}
  Set $\varphi_h=-\Delta_h U^n$ in \eqref{eq5.1} and obtain
 \begin{equation}
  (\bar{\partial}_tU^n,-\Delta_h U^n)+(1+\norm{\nabla U^n}^2)\norm{\Delta_h U^n}^2=(f^n,-\Delta_h U^n).
 \end{equation}
Note that
\begin{align*}
 (\bar{\partial}_tU^n,-\Delta_h U^n)=(\bar{\partial}_t\nabla U^n,\nabla U^n)=\frac{1}{2}\bar{\partial}_t\norm{\nabla U^n}^2+\frac{k}{2}\norm{\bar{\partial}_t\nabla U^n}^2.
\end{align*}
Therefore, 
\begin{align*}
 \frac{1}{2}\bar{\partial}_t\norm{\nabla U^n}^2+\frac{k}{2}\norm{\bar{\partial}_t\nabla U^n}^2+&(1+\norm{\nabla U^n}^2)\norm{\Delta_h U^n}^2\\
 &=(f^n,-\Delta_h U^n)\\
 &\leq \frac{1}{2}\norm{f^n}^2+\frac{1}{2}\norm{\Delta_h U^n}^2\leq \frac{1}{2}\norm{f^n}^2+\frac{1}{2}(1+\norm{\nabla U^n}^2)\norm{\Delta_h U^n}^2
\end{align*}
Consequently,
\begin{equation}\label{eqx4.1}
 \bar{\partial}_t\norm{\nabla U^n}^2+(1+\norm{\nabla U^n}^2)\norm{\Delta_h U^n}^2\leq \norm{f^n}^2,
\end{equation}
and hence $\norm{\nabla U^n}^2$ is bounded.
\end{remark}

\subsection{Existence and uniqueness of discrete solution}
\begin{theorem}\label{thm5.1}
(Brouwer's fixed point theorem) \cite{kesavan}. Let $H$ be a finite dimensional Hilbert space with inner product (.,.) and $\norm{.}.$ Let $g:H\rightarrow H$ be a continuous function. If there exist $R>0$ such that
$(g(z),z)>0$ $\forall z$ with $\norm{z}=R,$ then there exists $z^*\in H$ such that $\norm{z}\leq R$ and $g(z^*)=0.$
\end{theorem}
\begin{theorem}\label{thm5.2}
 Given $U^{n-1}$, the discrete problem \eqref{eq5.1} has a unique solution $U^n , n\geq 1$.
\end{theorem}
\begin{proof}
 Given $U^{n-1}$, define a function $\mathbb{F}:V_h\rightarrow V_h$ for a fixed $n$ by 
 \begin{align}\label{eq5.4}
  (\mathbb{F}(v),\varphi_h)=(v,\varphi_h)+k(1+\norm{\nabla v}^2)(\nabla v,\nabla \varphi_h)-k(f^n,\varphi_h)-(U^{n-1},\varphi_h)
 \end{align}
 Define a norm on $V_h$ as
 \begin{equation}
%   $\vertiii{a}$,
|||v|||=(\norm{v}^2+k\norm{\nabla v}^2)^\frac{1}{2},
 \end{equation}
then $\mathbb{F}$ is continuous by sequential criterion. Now substituting $\varphi_h=v$ in \eqref{eq5.4} to obtain
\begin{align*}
 (\mathbb{F}(v),v)&=\norm{v}^2+k(1+\norm{\nabla v}^2)\norm{\nabla v}^2-k(f^n,v)-(U^{n-1},v)\\
 &\geq \norm{v}^2+k\norm{\nabla v}^2-k(\norm{f^n}+\norm{U^{n-1}})\norm{v}
\end{align*}
Choosing $R$ in such a way that $|||v|||=R$ with $R-k(\norm{f^n}+\norm{U^{n-1}})>0$ and hence,
\begin{align*}
 (\mathbb{F}(v),v)\geq R(R-k(\norm{f^n}+\norm{U^{n-1}}))>0
\end{align*}
A use of theorem \ref{thm5.1} would provide us the existence of $\{U^n\}_{n\geq 1}.$\\
Now to prove uniqueness, set $W^n=U^n_1-U^n_2$, where $U^n_1$ and $U^n_2$ are the solutions of \eqref{eq5.1}.
Then, $W^n$ satisfy 
\begin{equation*}
 (\bar\partial_tW^n,\varphi_h)+\left((1+\norm{\nabla U^n_1})^2\nabla U^n_1-(1+\norm{\nabla U^n_2})^2\nabla U^n_2,\nabla\varphi_h\right)=0
\end{equation*}
Substitute $\varphi_h=W^n=U^n_1-U^n_2,$ we obtain
\begin{equation}
  (\bar\partial_tW^n,W^n)+\left((1+\norm{\nabla U^n_1})^2\nabla U^n_1-(1+\norm{\nabla U^n_2})^2\nabla U^n_2,\nabla(U^n_1-U^n_2)\right)=0
\end{equation}
Using monotonicity property in Lemma \ref{2.6} we observe
\begin{align*}
 \left((1+\norm{\nabla U^n_1})^2\nabla U^n_1-(1+\norm{\nabla U^n_2})^2\nabla U^n_2,\nabla(U^n_1-U^n_2)\right)\geq\norm{\nabla W^n}^2\geq 0.
\end{align*}
Consequently $$(\bar\partial_tW^n,W^n)\leq 0$$ and hence
\begin{equation}
 \frac{1}{2k}(\norm{W^n}^2-\norm{W^{n-1}})\leq 0.
\end{equation}
Taking summation from $n=1$ to $N$ to obtain
\begin{equation}
 \norm{W^N}^2\leq \norm{W^0}^2.
\end{equation}
Since $W^0=0$, it follows that $W^N=0$ which leads to a contradiction. Hence, the solution is unique. This completes the rest of the proof.
\end{proof}
\subsection{Error Analysis for Backward Euler Method}
In this subsection, we discuss error estimates for fully discrete finite element method.

Now spllit the error $e^n=u(t_n)-U^n=\big(u(t_n)-\widetilde u(t_n)\big)-\big(U^n-\widetilde u(t_n)\big)=\eta^n-\theta^n$, where $U^n$ is the solution of \eqref{eq5.1} and $u(t_n)$ is the soution of \eqref{temp1}, and
$\eta^n=\eta(t_n)$ is defined in \eqref{eq5}.

 Using \eqref{temp1} at $t=t_n$ and \eqref{eq5.1}, the equation in $\theta^n$ becomes for all $\phi_h \in V_h$
\begin{align}
 (\bar{\partial_t}\theta^n,\varphi_h)+\Big((1+\norm{\nabla U^n}^2)\nabla U^n-&(1+\norm{\nabla \widetilde u(t_n)}^2)\nabla \widetilde u(t_n),\nabla\varphi_h\Big)\notag\\
 =&(\bar{\partial_t}\eta^n,\varphi_h)+(u_t(t_n)-\bar{\partial_t}u(t_n),\varphi_h)\notag\\
 &+\Big((1+\norm{\nabla u(t_n)}^2)\nabla u(t_n)-(1+\norm{\nabla\widetilde u(t_n)}^2)\nabla\widetilde u(t_n),\nabla \varphi_h\Big)\label{eq7.3}.
\end{align}
\begin{theorem}\label{thm5.4}
Let  $0<\alpha<\dfrac{\lambda_1}{2}.$ Choose $k_0>0$ such that for $0<k\leq k_0$ 
\begin{equation}\label{er1}
 \Big(1+\dfrac{\lambda_1k}{2}\Big)>e^{\alpha k},
\end{equation}
where $\beta=\Big(e^{-\alpha k}-\dfrac{2}{k\lambda_1}(1-e^{-\alpha k})\Big)>0$ holds. Then, there exists a positive constant $C=C(\lambda_1, K_1)$ independent of $h$ and $k$ such that
 \begin{align}
\norm{\theta^N}^2+&k\beta e^{-2\alpha t_N}\sum_{n=1}^{N}e^{2\alpha t_n}\norm{\nabla\theta^n}^2\notag\\ 
&\leq C(\lambda_1,K_1)e^{-2\alpha t_N}(k^2+h^2)\Big(\norm{u_0}^2_{H^3}+\left((1+\norm{\nabla u_0}^2)^2\norm{u_0}^2_{H^2}+\norm{f_0}^2\right)\notag\\
&+\int_{0}^{t_N}e^{2\alpha s}\norm{f_s}^2_{-1}ds\Big).
\end{align}
\end{theorem}
\begin{proof}
Multiplying \eqref{eq7.3} by $e^{\alpha t_n}$ and putting $\varphi_h=e^{\alpha t_n}\theta^n=\widehat\theta^n$, we obtain 
\begin{align}
 (e^{\alpha t_n}\bar{\partial_t}\theta^n,\widehat\theta^n)+e^{\alpha t_n}&\Big((1+\norm{\nabla U^n}^2)\nabla U^n-(1+\norm{\nabla\widetilde u(t_n)}^2)\nabla\widetilde u(t_n),\nabla\widehat\theta^n\Big)\notag\\
 &=e^{\alpha t_n}(\bar{\partial_t}\eta^n,\widehat\theta^n)+e^{\alpha t_n}(u_t(t_n)-\bar{\partial_t}u(t_n),\widehat\theta^n)\notag\\
 &+e^{\alpha t_n}\Big((1+\norm{\nabla u(t_n)}^2)\nabla u(t_n)-(1+\norm{\nabla\widetilde u(t_n)}^2)\nabla \widetilde u(t_n),\nabla\widehat\theta^n\Big)\label{eq7.4}.
\end{align}
Note that
\begin{equation}\label{eq7.5}
 e^{\alpha t_n}\bar{\partial_t}\theta^n=e^{\alpha k}\bar{\partial_t}\widehat\theta^n-\frac{(e^{\alpha k}-1)}{k}\widehat\theta^n.
\end{equation}
Now by monotonicity property given in Lemma \ref{2.6}
\begin{equation}\label{eq7.6}
 e^{2\alpha t_n}\Big((1+\norm{\nabla U^n}^2)\nabla U^n-(1+\norm{\nabla\widetilde u(t_n)}^2)\nabla\widetilde u(t_n),\nabla(U^n-\widetilde u(t_n))\Big)\geq \norm{\nabla\widehat\theta^n}^2.
\end{equation}
Therefore, using \eqref{eq7.5}-\eqref{eq7.6} and Ritz-projection, we obtain from \eqref{eq7.4}
\begin{align}
 e^{\alpha k}(\bar{\partial_t}\widehat\theta^n,\widehat\theta^n)-&\frac{(e^{\alpha k}-1)}{k}\norm{\widehat\theta^n}^2+\norm{\nabla\widehat\theta^n}^2\label{eq7.7}\\
 &\leq e^{\alpha t_n}(\bar{\partial_t}\eta^n,\widehat\theta^n)+e^{\alpha t_n}(u_t(t_n)-\bar{\partial_t}u(t_n),\widehat\theta^n)\notag\\
 &+e^{\alpha t_n}\Big((\norm{\nabla u(t_n)}^2-\norm{\nabla\widetilde u(t_n)}^2)\nabla u(t_n),\nabla\widehat\theta^n\Big).\notag
\end{align}
Note that
\begin{equation}\label{eq7.8}
 (\bar{\partial_t}\widehat\theta^n,\widehat\theta^n)=\frac{1}{2}\bar{\partial_t}\norm{\widehat\theta^n}^2+\frac{k}{2}\norm{\bar{\partial_t}\widehat\theta^n}^2\geq\frac{1}{2}\bar{\partial_t}\norm{\widehat\theta^n}^2
\end{equation}
 and
 \begin{equation}\label{eq7.9}
  \norm{\widehat\theta^n}^2\leq\dfrac{1}{\lambda_1}\norm{\nabla\widehat\theta^n}^2.
 \end{equation}
Now, using Poincare's inequality and Young's inequality, we estimate the first, second and third terms in the right hand side of \eqref{eq7.7} as follows
\begin{equation}\label{eq7.10}
 e^{\alpha t_n}(\bar{\partial_t}\eta^n,\widehat\theta^n)\leq \frac{3e^{2\alpha t_n}}{2\lambda_1}\norm{\bar{\partial_t}\eta^n}^2+\frac{1}{6}\norm{\nabla\widehat\theta^n}^2,
\end{equation}
\begin{equation}\label{eq7.11}
 e^{\alpha t_n}(u_t(t_n)-\bar{\partial_t}u(t_n),\widehat\theta^n)\leq\frac{3}{2\lambda_1}e^{2\alpha t_n}\norm{u_t(t_n)-\bar{\partial_t}u(t_n)}^2+\frac{1}{6}\norm{\nabla\widehat\theta^n}^2,
\end{equation}
and
\begin{align}\label{eq7.12}
e^{\alpha t_n}\Big((\norm{\nabla u(t_n)}^2-\norm{\nabla\widetilde u(t_n)}^2)\nabla u(t_n),\nabla\widehat\theta^n\Big)\leq6e^{2\alpha t_n}\norm{\nabla u(t_n)}^4\norm{\nabla\eta^n}^2+\frac{1}{6}\norm{\nabla\widehat\theta^n}^2. 
\end{align}
Therefore, using \eqref{eq7.8}-\eqref{eq7.12} in \eqref{eq7.7} and multiplying by $2e^{-\alpha k}$ the resulting inequality, we arrive at
\begin{align}
 \bar{\partial_t}\norm{\widehat\theta^n}^2+\Big(e^{-\alpha k}-\dfrac{2}{k\lambda_1}(1-e^{-\alpha k})\Big)\norm{\nabla\widehat\theta^n}^2&\leq \dfrac{3}{2\lambda_1}e^{-\alpha k}e^{2\alpha t_n}\Big(\norm{\bar{\partial_t}\eta^n}^2+\norm{u_t(t_n)-\bar{\partial_t}u(t_n)}^2\Big)\notag\\
 &+6e^{-\alpha k}e^{2\alpha t_n}\norm{\nabla u(t_n)}^4\norm{\nabla\eta^n}^2.\label{eq7.13}
\end{align}
With $0<\alpha<\dfrac{\lambda_1}{2}$, choose $k_0>0$ such that for $0<k\leq k_0,$ \eqref{er1} is satisfied. Then $\beta=\Big(e^{-\alpha k}-\dfrac{2}{k\lambda_1}(1-e^{-\alpha k})\Big)>0$. 
Therefore, multiplying \eqref{eq7.13} by $k$, and summing over $n=1$ to $N$, we arrive at
\begin{align}
 \norm{\widehat\theta^N}^2+k\beta\sum_{n=1}^{N}\norm{\nabla\widehat\theta^n}^2\leq\norm{\widehat\theta^0}^2&+\dfrac{3}{2\lambda_1}ke^{-\alpha k}\sum_{n=1}^{N}e^{2\alpha t_n}\Big(\norm{\bar{\partial_t}\eta^n}^2+\norm{u_t(t_n)-\bar{\partial_t}u(t_n)}^2\Big)\notag\\
 &+6ke^{-\alpha k}\sum_{n=1}^{N}e^{2\alpha t_n}\norm{\nabla u(t_n)}^4\norm{\nabla\eta^n}^2.\label{eq7.14}
\end{align}
Note that
\begin{align}
 \norm{\bar{\partial_t}\eta^n}^2=&\frac{1}{k^2}\Big(\int_{t_{n-1}}^{t_n}(u-\widetilde u)_t ds\Big)^2\notag\\
 &\leq\frac{1}{k}\int_{t_{n-1}}^{t_n}\norm{\eta_t}^2 ds\leq\frac{1}{k}Ch^2\int_{t_{n-1}}^{t_n}\norm{\nabla u_t(s)}^2 ds\label{ex1}.
\end{align}
Therefore, the second term on the right hand side of \eqref{eq7.14} can be bounded by
\begin{align}
 \dfrac{3}{2\lambda_1}ke^{-\alpha k}\sum_{n=1}^{N}e^{2\alpha t_n}\norm{\bar{\partial_t}\eta^n}^2&\leq C(\lambda_1)h^2e^{-\alpha k}\sum_{n=1}^{N}\int_{t_{n-1}}^{t_n}e^{2\alpha t_n}\norm{\nabla u_t(s)}^2 ds\notag\\
 &=C(\lambda_1)h^2e^{-\alpha k}e^{2\alpha k}\sum_{n=1}^{N}\int_{t_{n-1}}^{t_n}e^{2\alpha t_{n-1}}\norm{\nabla u_t(s)}^2 ds\notag\\
 &\leq C(\lambda_1)h^2\int_{0}^{t_N}e^{2\alpha s}\norm{\nabla u_t(s)}^2 ds.\label{eq7.15}
\end{align}
By the Taylor series expansion of $u$ around $t_n$ in the interval $(t_{n-1},t_n)$, we obtain
\begin{align}
 \norm{u_t(t_n)-\bar{\partial_t}u(t_n)}^2&\leq \frac{1}{k^2}\Big(\int_{t_{n-1}}^{t_n}(t_n-s)\norm{u_{tt}(s)} ds\Big)^2\notag\\
 &\leq\frac{1}{k^2}\Big(\int_{t_{n-1}}^{t_n}(t_n-s)^2 ds\Big)\Big(\int_{t_{n-1}}^{t_n}\norm{u_{tt}(s)}^2 ds\Big)\notag\\
 &=\frac{k}{3}\int_{t_{n-1}}^{t_n}\norm{u_{tt}(s)}^2 ds\label{ex2},
\end{align}
and the third term on the right hand of \eqref{eq7.14} side is now bounded by
\begin{align}
\dfrac{3}{2\lambda_1}k\sum_{n=1}^{N}e^{2\alpha t_n}\norm{u_t(t_n)-\bar{\partial_t}u(t_n)}^2 &\leq\frac{1}{2\lambda_1}k^2e^{-\alpha k}\sum_{n=1}^{N}\int_{t_{n-1}}^{t_n}e^{2\alpha t_n}\norm{u_{tt}(s)}^2 ds\notag\\
&\leq\frac{1}{2\lambda_1}k^2e^{-\alpha k}e^{2\alpha k}\sum_{n=1}^{N}\int_{t_{n-1}}^{t_n}e^{2\alpha s}\norm{u_{tt}(s)}^2 ds\notag\\
&\leq C(\lambda_1)k^2\int_{0}^{t_N}e^{2\alpha s}\norm{u_{tt}(s)}^2 ds.\label{eq8.1}
\end{align}
For the last term on the right hand side of \eqref{eq7.14} bound is 
\begin{align}
 6ke^{-\alpha k}\sum_{n=1}^{N}e^{2\alpha t_n}\norm{\nabla u(t_n)}^4\norm{\nabla\eta^n}^2&\leq Ch^2e^{-\alpha k}k\sum_{n=1}^{N}e^{2\alpha t_n}\norm{\nabla u(t_n)}^4\norm{u(t_n)}^2_{H^2}\notag\\
 &\leq Ch^2e^{-\alpha k}K_1\Big(k\sum_{n=1}^{N}\norm{\widehat u(t_n)}^2_{H^2}\Big).
\end{align}
Therefore, from \eqref{eq7.14} we arrive at
\begin{align}
 \norm{\widehat\theta^N}^2+k\beta\sum_{n=1}^{N}\norm{\nabla\widehat\theta^n}^2&\leq\norm{\theta^0}^2+C(\lambda_1)h^2\int_{0}^{t_N}e^{2\alpha s}\norm{\nabla u_t(s)}^2 ds
 +C(\lambda_1)k^2\int_{0}^{t_N}e^{2\alpha s}\norm{u_{tt}(s)}^2 ds\notag\\
 &+C(K_1)h^2e^{-\alpha k}\norm{u(t_n)}^2_{H^2}\Big(k\sum_{n=1}^{N}e^{2\alpha t_n}\Big).\label{eq8.2}
\end{align}
With a choice $U^0=\widetilde u(0)$, $\theta^0=0$. But with $U^0=P_hu_0$,
$$\norm{\theta^0}\leq Ch\norm{\nabla u_0}.$$
Multiply \eqref{eq8.2} by $e^{-2\alpha t_N}$ to obtain
\begin{align}
 \norm{\theta^N}^2+k\beta e^{-2\alpha t_N}\sum_{n=1}^{N}e^{2\alpha t_n}\norm{\nabla\theta^n}^2&\leq e^{-2\alpha t_N}\norm{\theta^0}^2+C(\lambda_1)e^{-2\alpha t_N}h^2\int_{0}^{t_N}e^{2\alpha s}\norm{\nabla u_t(s)}^2 ds\notag\\
 &+C(\lambda_1)e^{-2\alpha t_N}k^2\int_{0}^{t_N}e^{2\alpha s}\norm{u_{tt}(s)}^2 ds\notag\\
 &+C(K_1)e^{-\alpha k}\Big(e^{-2\alpha t_N}k\sum_{n=1}^{N}e^{2\alpha t_n}\Big)\norm{u(t_n)}^2_{H^2}h^2.\label{eq8.3}
\end{align}
Note that
$$e^{-2\alpha t_N}\Big(k\sum_{n=1}^{N}e^{2\alpha t_n}\Big)=e^{-2\alpha t_N}\frac{1}{e^{2\alpha k}-1}ke^{2\alpha k}(e^{2\alpha t_N}-1)\leq C.$$
By using Lemma \ref{2.4} and \ref{nl2}, the second and third term in the right hand side of \eqref{eq8.3} are bounded respectively.\\
Therefore
\begin{align}
\norm{\theta^N}^2+&k\beta e^{-2\alpha t_N}\sum_{n=1}^{N}e^{2\alpha t_n}\norm{\nabla\theta^n}^2\notag\\ 
&\leq C(\lambda_1,K_1)e^{-2\alpha t_N}(k^2+h^2)\Big(\norm{u_0}^2_{H^3}+\left((1+\norm{\nabla u_0}^2)^2\norm{u_0}^2_{H^2}+\norm{f_0}^2\right)\notag\\
&+\int_{0}^{t_N}e^{2\alpha s}\norm{f_s}^2_{-1}ds\Big).
\end{align}
This completes the rest of the proof.
\end{proof}

% ------------------------------------
Since from \eqref{eq3.4} $\norm{\nabla\eta^n}\leq Ch\norm{u(t_n)}_{H^2}$, is known,to in order to, estimate of $\norm{\nabla u(t_n)-U^n}$, it is enough to estimate $\norm{\nabla\theta^n}$.
\begin{theorem}\label{thm5.6}
Assume that $0<\alpha<\dfrac{\lambda_1}{2}$ and choose $k_0>0$ be such that for $0<k\leq k_0$, \eqref{er1} is true.
Then, there exists a positive constant $C=C(\alpha, K)$ such that
 \begin{align}
 ke^{-2\alpha t_N}\sum_{n=1}^{N}&\norm{\bar{\partial_t}\widehat\theta^n}^2+e^{-\alpha k}(1+\norm{\nabla U^N}^2)\norm{\nabla\theta^N}^2\notag\\
 &\leq C(\alpha,K)e^{-2\alpha t_N}(h^2+k^2)\Big(\norm{u_0}^2_{H^3}+\left((1+\norm{\nabla u_0}^2)^2\norm{u_0}^2_{H^2}+\norm{f_0}^2\right)\notag\\
&+\int_{0}^{t_N}e^{2\alpha s}\norm{f_s}^2_{-1}ds\Big),
\end{align}
where $\beta=\Big(e^{-\alpha k}-\dfrac{2}{k\lambda_1}(1-e^{-\alpha k})\Big)>0$, and $K$ depends on $\nabla u_0$.
\end{theorem}
\begin{proof}
Multiply the equation \eqref{eq7.3} by $e^{\alpha t_n}$ and then putting $\varphi_h=\bar{\partial_t}\widehat\theta^n$, we obtain
\begin{align}
 (e^{\alpha t_n}\bar{\partial_t}\theta^n,\bar{\partial_t}\widehat\theta^n)+e^{\alpha t_n}&\Big((1+\norm{\nabla U^n}^2)\nabla U^n-(1+\norm{\nabla \widetilde u(t_n)}^2)\nabla \widetilde u(t_n),\nabla\bar{\partial_t}\widehat\theta^n\Big)\notag\\
 &=e^{\alpha t_n}(\bar{\partial_t}\eta^n,\bar{\partial_t}\widehat\theta^n)+e^{\alpha t_n}(u_t(t_n)-\bar{\partial_t}u(t_n),\bar{\partial_t}\widehat\theta^n)\notag\\
 &+e^{2\alpha t_n}\Big((1+\norm{\nabla u(t_n)}^2)\nabla u(t_n)-(1+\norm{\nabla\widetilde u(t_n)}^2)\nabla\widetilde u(t_n),\nabla\bar{\partial_t}\widehat\theta^n\Big).\label{eq9.2}
\end{align}
Using \eqref{eq7.5} in \eqref{eq9.2}, we find that
\begin{align}
 e^{\alpha k}\norm{\bar{\partial_t}\widehat\theta^n}^2+e^{\alpha t_n}&\Big((1+\norm{\nabla U^n}^2)\nabla U^n-(1+\norm{\nabla \widetilde u(t_n)}^2)\nabla \widetilde u(t_n),\nabla\bar{\partial_t}\widehat\theta^n\Big)\notag\\
 &=\frac{(e^{\alpha k}-1)}{k}(\widehat\theta^n,\bar{\partial_t}\widehat\theta^n)+e^{\alpha t_n}(\bar{\partial_t}\eta^n,\bar{\partial_t}\widehat\theta^n)+e^{\alpha t_n}(u_t(t_n)-\bar{\partial_t}u(t_n),\bar{\partial_t}\widehat\theta^n)\notag\\
 &+e^{2\alpha t_n}\Big((1+\norm{\nabla u(t_n)}^2)\nabla u(t_n)-(1+\norm{\nabla\widetilde u(t_n)}^2)\nabla\widetilde u(t_n),\nabla\bar{\partial_t}\widehat\theta^n\Big).\label{eq9.3}
\end{align}
 For the second term on the left hand side of \eqref{eq9.3}, use Ritz projection to rewrite it as
\begin{align}
 e^{\alpha t_n}\Big((1+\norm{\nabla U^n}^2)\nabla U^n-&(1+\norm{\nabla \widetilde u(t_n)}^2)\nabla \widetilde u(t_n),\nabla\bar{\partial_t}\widehat\theta^n\Big)\notag\\
 &=(\nabla\widehat\theta^n,\nabla\bar{\partial_t}\widehat\theta^n)+\norm{\nabla U^n}^2(\nabla\widehat\theta^n,\nabla\bar{\partial_t}\widehat\theta^n)\notag\\
 &-e^{\alpha t_n}\Big((\norm{\nabla U^n}^2-\norm{\nabla\widetilde u(t_n)}^2)\Delta u(t_n),\bar{\partial_t}\widehat\theta^n\Big).\label{eq9.4}
\end{align}
Note that
\begin{align}
 (\nabla\widehat\theta^n,\nabla\bar{\partial_t}\widehat\theta^n)=(\nabla\widehat\theta^n,\frac{\nabla\widehat\theta^n-\nabla\widehat\theta^{n-1}}{k})\geq\frac{1}{2}\bar{\partial_t}\norm{\nabla\widehat\theta^n}^2.\label{eq9.5}
\end{align}
The fourth term on the right hand side of \eqref{eq9.3}, can be bounded using Ritz projection as
\begin{align}
 e^{\alpha t_n}\Big((1+\norm{\nabla u(t_n)}^2)\nabla u(t_n)-&(1+\norm{\nabla\widetilde u(t_n)}^2)\nabla\widetilde u(t_n),\nabla\bar{\partial_t}\widehat\theta^n\Big)\notag\\
 &=e^{\alpha t_n}\Big(\norm{\nabla u(t_n)}^2\nabla u(t_n)-\norm{\nabla\widetilde u(t_n)}^2\nabla\widetilde u(t_n),\nabla\bar{\partial_t}\widehat\theta^n\Big)\notag\\
 &=e^{\alpha t_n}\Big((\norm{\nabla u(t_n)}^2-\norm{\nabla\widetilde u(t_n)}^2)\nabla\widetilde u(t_n),\nabla\bar{\partial_t}\widehat\theta^n\Big)\notag\\
 &=-e^{\alpha t_n}\Big((\norm{\nabla u(t_n)}^2-\norm{\nabla\widetilde u(t_n)}^2)\Delta u(t_n),\bar{\partial_t}\widehat\theta^n\Big)\notag\\
 &\leq Ce^{\alpha t_n}\norm{\nabla u(t_n)}\norm{\nabla\eta^n}\Delta u(t_n)\norm{\bar{\partial_t}\widehat\theta^n}.\label{eq9.6}
\end{align}
Therefore, using \eqref{eq9.4}-\eqref{eq9.6} in \eqref{eq9.3} and then multiplying the resulting inequality by $e^{-\alpha k}$, we obtain
\begin{align}
 \norm{\bar{\partial_t}\widehat\theta^n}^2+\frac{1}{2}e^{-\alpha k}(1+\norm{\nabla U^n}^2)\bar{\partial_t}\norm{\nabla\widehat\theta^n}^2&\leq\frac{(1-e^{-\alpha k})}{k}\norm{\widehat\theta^n}\norm{\bar{\partial_t}\widehat\theta^n}+e^{-\alpha k}e^{\alpha t_n}\norm{\bar{\partial_t}\eta^n}\norm{\bar{\partial_t}\widehat\theta^n}\notag\\
 &+e^{-\alpha k}e^{\alpha t_n}\norm{u_t(t_n)-\bar{\partial_t}u(t_n)}\norm{\bar{\partial_t}\widehat\theta^n}\notag\\
 &+Ce^{-\alpha k}e^{\alpha t_n}\norm{\nabla u(t_n)}\norm{\Delta u(t_n)}\norm{\nabla\eta^n}\norm{\bar{\partial_t}\widehat\theta^n}\notag\\
 &+e^{-\alpha k}e^{\alpha t_n}(\norm{\nabla U^n}+\norm{\nabla u(t_n)})\norm{\nabla\theta^n}\norm{\Delta u(t_n)}\norm{\bar{\partial_t}\widehat\theta^n}.\label{eq9.7}
\end{align}
Apply Young's inequality $ab\leq\frac{\epsilon}{2}a^2+\frac{1}{2\epsilon}b^2$ with $\epsilon=5$ and also $\frac{(1-e^{-\alpha k})}{k}=\alpha e^{-\alpha k^*}$ for some $k^*\in(0,k)$ and then multiply the resulting inequality by $2$ to obtain
\begin{align}
 \norm{\bar{\partial_t}\widehat\theta^n}^2+e^{-\alpha k}(1+\norm{\nabla U^n}^2)\bar{\partial_t}\norm{\nabla\widehat\theta^n}^2&\leq C(\alpha)\Big(e^{-2\alpha k^*}\norm{\widehat\theta^n}^2+e^{-2\alpha k}e^{2\alpha t_n}\norm{\bar{\partial_t}\eta^n}^2\notag\\
 &+e^{-2\alpha k}e^{2\alpha t_n}\norm{u_t(t_n)-\bar{\partial_t}u(t_n)}^2\notag\\
 &+e^{-2\alpha k}e^{2\alpha t_n}\norm{\nabla u(t_n)}^2\norm{\Delta u(t_n)}^2\norm{\nabla\eta^n}^2\notag\\
 &+e^{-2\alpha k}e^{2\alpha t_n}(\norm{\nabla U^n}^2+\norm{\nabla u(t_n)}^2)\norm{\nabla\theta^n}^2\norm{\Delta u(t_n)}^2\Big).\label{eq9.8}
\end{align}
On multiplying \eqref{eq9.8} by $ke^{-2\alpha t_N}$ and summing over $n=1$ to $N$, using \eqref{ex1}, \eqref{ex2} and the estimate $\norm{\nabla \eta}$ we arrive at
\begin{align}
 ke^{-2\alpha t_N}\sum_{n=1}^{N}&\norm{\bar{\partial_t}\widehat\theta^n}^2+e^{-\alpha k}(1+\norm{\nabla U^N}^2)\norm{\nabla\theta^N}^2\notag\\
 &\leq e^{-\alpha k}e^{-2\alpha t_N}(1+\norm{\nabla U^1}^2)\norm{\nabla \theta^0}^2+e^{-\alpha k}e^{-2\alpha t_N}k\sum_{n=1}^{N-1}\bar{\partial}_t\norm{\nabla U^{n+1}}^2\norm{\nabla\widehat\theta^n}^2\notag\\
 &+C(\alpha)e^{-2\alpha t_N}k\sum_{n=1}^{N}\norm{\widehat\theta^n}^2+C(\alpha)h^2e^{-2\alpha t_N}\int_{0}^{t_N}e^{2\alpha s}\norm{\nabla u_t(s)}^2 ds\notag\\
 &+C(\alpha)e^{-2\alpha t_N}k^2\int_{0}^{t_N}e^{2\alpha s}\norm{u_{tt}(s)}^2 ds\notag\\
 &+e^{-2\alpha k}e^{-2\alpha t_N}h^2k\sum_{n=1}^{N}e^{2\alpha t_n}\norm{\nabla u(t_n)}^2\norm{\Delta u(t_n)}^4\notag\\
 &+e^{-2\alpha k}e^{-2\alpha t_N}k\sum_{n=1}^{N}e^{2\alpha t_n}(\norm{U^n}^2+\norm{u(t_n)}^2)\norm{\Delta u(t_n)}^2\norm{\nabla\theta^n}^2\label{eq9.9}
\end{align}
From remark \ref{rm4.1}, Lemma \ref{4.1}, $e^{-2\alpha t_N}\Big(k\sum_{n=1}^{N}e^{2\alpha t_n}\Big)\leq C$ 
and the previous theorem \ref{thm5.4}, the first term, second term, third term and last term on the right hand side of equation \eqref{eq9.9} are bounded. From Lemmas \ref{2.2}, \ref{2.4}, \ref{2.5}, \ref{nl2} the other terms on the right hand side of equation \eqref{eq9.9} are bounded.
Therefore, we arrive at
\begin{align}
 ke^{-2\alpha t_N}\sum_{n=1}^{N}&\norm{\bar{\partial_t}\widehat\theta^n}^2+e^{-\alpha k}(1+\norm{\nabla U^N}^2)\norm{\nabla\theta^N}^2\notag\\
 &\leq C(\alpha,K)e^{-2\alpha t_N}(h^2+k^2)\Big(\norm{u_0}^2_{H^3}+\left((1+\norm{\nabla u_0}^2)^2\norm{u_0}^2_{H^2}+\norm{f_0}^2\right)\notag\\
&+\int_{0}^{t_N}e^{2\alpha s}\norm{f_s}^2_{-1} ds\Big).
\end{align}
This completes the rest of the proof.
\end{proof}
Since at each time level, we need to solve the system of nonlinear equation, below, we discuss modified backward Euler method which gives rise to a system of linear equations at each time step.
\subsection{ Modified Backward Euler Method}
For $n\geq 1$ and given $U^{n-1}$, the fully discrete linear scheme based on backward Euler method is to seek $U^n \in V_h$  as a solution of
\begin{align}
&(\bar{\partial}_tU^n,\varphi)+(1+\norm{\nabla U^{n-1}}^2)(\nabla U^n,\nabla \varphi)=(f^n,\varphi) \quad \forall \varphi \in V_h,\label{x1}\\
  &U^0=u_{0h}.\notag
 \end{align}
 At each time level using {\it{a priori}} bound of $U^n$, this system of linear equation has a unique solution. Now for the error analysis, split the error
 $e^n=u(t_n)-U^n=\big(u(t_n)-\widetilde u(t_n)\big)-\big(U^n-\widetilde u(t_n)\big)=\eta^n-\theta^n$, the equation in $\theta^n$ becomes
\begin{align}
 (\bar{\partial_t}\theta^n,\varphi_h)+\Big((1+\norm{\nabla U^{n-1}}^2)\nabla U^n-&(1+\norm{\nabla \widetilde u(t_{n-1})}^2)\nabla \widetilde u(t_n),\nabla\varphi_h\Big)\notag\\
 =&(\bar{\partial_t}\eta^n,\varphi_h)+(u_t(t_n)-\bar{\partial_t}u(t_n),\varphi_h)\notag\\
 &+\Big((1+\norm{\nabla u(t_n)}^2)\nabla u(t_n)-(1+\norm{\nabla\widetilde u(t_{n-1})}^2)\nabla\widetilde u(t_n),\nabla \varphi_h\Big)\label{x2}.
\end{align}
\begin{theorem}\label{thm5.7}
 Assume that $0<\alpha<\dfrac{\lambda_1}{2}$ and choose $k_0>0$ such that for $0<k\leq k_0$ \eqref{er1} is true.
 Then, there exists a positive constant $C=C(\lambda_1, K_1)$ independent of $h$ and $k$ such that
 \begin{align}
\norm{\theta^N}^2+&k\beta e^{-2\alpha t_N}\sum_{n=1}^{N}e^{2\alpha t_n}\norm{\nabla\theta^n}^2\notag\\ 
&\leq C(\lambda_1,K_1)e^{-2\alpha t_N}(k^2+h^2)\Big(\norm{u_0}^2_{H^3}+\left((1+\norm{\nabla u_0}^2)^2\norm{u_0}^2_{H^2}+\norm{f_0}^2\right)\notag\\
&+\int_{0}^{t_N}e^{2\alpha s}\norm{f_s}^2_{-1}ds\Big),
\end{align}
 where $\beta=\Big(e^{-\alpha k}-\dfrac{2}{k\lambda_1}(1-e^{-\alpha k})\Big)>0$.
\end{theorem}
\begin{proof}
 Multiplying \eqref{x2} by $e^{\alpha t_n}$ and putting $\varphi_h=e^{\alpha t_n}\theta^n=\widehat\theta^n$, we obtain
 \begin{align}
 (e^{\alpha t_n}\bar{\partial_t}\theta^n,\widehat\theta^n)+e^{\alpha t_n}&\Big((1+\norm{\nabla U^{n-1}}^2)\nabla U^n-(1+\norm{\nabla\widetilde u(t_{n-1})}^2)\nabla\widetilde u(t_n),\nabla\widehat\theta^n\Big)\notag\\
 &=e^{\alpha t_n}(\bar{\partial_t}\eta^n,\widehat\theta^n)+e^{\alpha t_n}(u_t(t_n)-\bar{\partial_t}u(t_n),\widehat\theta^n)\notag\\
 &+e^{\alpha t_n}\Big((1+\norm{\nabla u(t_n)}^2)\nabla u(t_n)-(1+\norm{\nabla\widetilde u(t_{n-1})}^2)\nabla \widetilde u(t_n),\nabla\widehat\theta^n\Big)\label{x4}.
\end{align}
The second term on the left hand side can be written as
\begin{align}
 e^{2\alpha t_n}\Big((1+\norm{\nabla U^{n-1}}^2)&\nabla U^n-(1+\norm{\nabla\widetilde u(t_{n-1})}^2)\nabla\widetilde u(t_n),\nabla\theta^n\Big)\notag\\
 &=e^{2\alpha t_n}\Big((1+\norm{\nabla U^{n}}^2)\nabla U^n-(1+\norm{\nabla\widetilde u(t_{n})}^2)\nabla\widetilde u(t_n),\nabla\theta^n\Big)\notag\\
 &+e^{\alpha t_n}\Big((\norm{\nabla U^{n-1}}^2-\norm{\nabla U^n}^2)\nabla U^n,\nabla\widehat\theta^n\Big)\notag\\
 &+e^{\alpha t_n}\Big((\norm{\nabla\widetilde u(t_n)}^2-\norm{\nabla\widetilde u(t_{n-1})}^2)\nabla\widetilde u(t_n),\nabla\widehat\theta^n\Big)\notag\\
 &\geq\norm{\nabla\widehat\theta^n}^2-Ce^{2\alpha t_n}k^2(\bar{\partial_t}\norm{\nabla U^n}^2)^2\norm{\nabla U^n}^2-\frac{1}{12}\norm{\nabla\widehat\theta^n}^2\notag\\
 &-e^{\alpha t_n}\Big((\norm{\nabla\widetilde u(t_n)}^2-\norm{\nabla\widetilde u(t_{n-1})}^2)\nabla\widetilde u(t_n),-\nabla\widetilde\theta^n\Big).\label{x5}
\end{align}
Note that
\begin{align}
 \norm{\nabla\widetilde u(t_n)-\nabla\widetilde u(t_{n-1})}^2&=\norm{\int_{t_{n-1}}^{t_n}(\nabla\widetilde u(t))_t dt}^2\notag\\
 &\leq C(k\int_{t_{n-1}}^{t_n}\norm{\nabla u_t}^2 dt+h^2k\int_{t_{n-1}}^{t_n}\norm{\Delta u_t}^2 dt)\label{ex5}.
\end{align}
The third term on the right hand side is bounded by
\begin{align}
e^{\alpha t_n}\Big((1+\norm{\nabla u(t_n)}^2)\nabla u(t_n)-&(1+\norm{\nabla\widetilde u(t_{n-1})}^2)\nabla \widetilde u(t_n),\nabla\widehat\theta^n\Big)\notag\\
&=e^{\alpha t_n}\Big((\norm{\nabla u(t_n)}^2-\norm{\nabla\widetilde u(t_{n})}^2)\nabla u(t_n),\nabla\widehat\theta^n\Big)\notag\\
&+e^{\alpha t_n}\Big((\norm{\nabla\widetilde u(t_n)}^2-\norm{\nabla\widetilde u(t_{n-1})}^2)\nabla u(t_n),\nabla\widehat\theta^n\Big)\notag\\
&\leq Ce^{2\alpha t_n}h^2\norm{\nabla u(t_n)}^4\norm{\Delta u}^2+\frac{1}{6}\norm{\nabla\widehat\theta^n}^2\notag\\
&+\Big(Ce^{2\alpha t_n}(\norm{\nabla u(t_n)}^2+\norm{\nabla u(t_{n-1})}^2)\norm{\nabla u(t_n)}^2&\notag\\
&\Big(k\int_{t_{n-1}}^{t_n}\norm{\nabla u_t}^2 dt+h^2k\int_{t_{n-1}}^{t_n}\norm{\Delta u_t}^2 dt\Big)\Big).\label{x6}
\end{align}
Now, continuing as before in Theorm \ref{thm5.4} we arrive at
 \begin{align}
 \bar{\partial_t}\norm{\widehat\theta^n}^2+\Big(e^{-\alpha k}-\dfrac{2}{k\lambda_1}(1-e^{-\alpha k})\Big)\norm{\nabla\widehat\theta^n}^2&\leq\dfrac{3}{2\lambda_1}e^{-\alpha k}e^{2\alpha t_n}(\norm{\bar{\partial_t}\eta^n}^2+\norm{u_t(t_n)-\bar{\partial_t}u(t_n)}^2)\notag\\
 &+Ce^{2\alpha t_n}k^2(\bar{\partial_t}\norm{\nabla U^n}^2)^2\norm{\nabla U^n}^2+Ce^{2\alpha t_n}h^2\norm{\nabla u(t_n)}^4\norm{\Delta u}^2\notag\\
 &+\frac{1}{6}\norm{\nabla\widehat\theta^n}^2+\Big(Ce^{2\alpha t_n}(\norm{\nabla u(t_n)}^2+\norm{\nabla u(t_{n-1})}^2)\norm{\nabla u(t_n)}^2\notag\\
 &\Big(k\int_{t_{n-1}}^{t_n}\norm{\nabla u_t}^2 dt+h^2k\int_{t_{n-1}}^{t_n}\norm{\Delta u_t}^2 dt\Big)\Big).\label{x7}
\end{align}
Mutiply \eqref{x7} by $k$, and summing over $n=1$ to $N$, we arrive at
\begin{align}
 \norm{\widehat\theta^N}^2+k\beta\sum_{n=1}^{N}\norm{\nabla\widehat\theta^n}^2&\leq\norm{\theta^0}^2+C(\lambda_1)(h^2+k^2)\int_{0}^{t_N}e^{2\alpha s}\norm{\nabla u_t(s)}^2 ds
 +C(\lambda_1)k^2\int_{0}^{t_N}e^{2\alpha s}\norm{u_{tt}(s)}^2 ds\notag\\
 &+C(\lambda_1)h^2k^2\int_{0}^{t_N}e^{2\alpha s}\norm{\Delta u_{t}(s)}^2 ds+C(K_1)k^2e^{-\alpha k}\norm{f^n}^2_{L^\infty(L^2)}\Big(k\sum_{n=1}^{N}e^{2\alpha t_n}\Big).\label{x8}
\end{align}
Proceed as in theorm \ref{thm5.4} to complete the rest of the proof.
\end{proof}
\begin{theorem}\label{thm5.8}
Assume that $0<\alpha<\dfrac{\lambda_1}{2}$ and choose $k_0>0$ be such that for $0<k\leq k_0$, \eqref{er1} is true.
Then, there exists a positive constant $C=C(\alpha, K)$ holds
 \begin{align}
 ke^{-2\alpha t_N}\sum_{n=1}^{N}&\norm{\bar{\partial_t}\widehat\theta^n}^2+e^{-\alpha k}(1+\norm{\nabla U^{N-1}}^2)\norm{\nabla\theta^N}^2\notag\\
 &\leq C(\alpha,K_1)e^{-2\alpha t_N}(h^2+k^2)\Big(\norm{u_0}^2_{H^3}+\left((1+\norm{\nabla u_0}^2)^2\norm{u_0}^2_{H^2}+\norm{f_0}^2\right)\notag\\
&+\int_{0}^{t_N}e^{2\alpha s}\norm{f_s}^2_{-1}ds\Big),
\end{align}
where $\beta=\Big(e^{-\alpha k}-\dfrac{2}{k\lambda_1}(1-e^{-\alpha k})\Big)>0$.
\end{theorem}
\begin{proof}
Multiply the equation \eqref{x2} by $e^{\alpha t_n}$ and then putting $\varphi_h=\bar{\partial_t}\widehat\theta^n$, we obtain
\begin{align}
  e^{\alpha k}\norm{\bar{\partial_t}\widehat\theta^n}^2+e^{\alpha t_n}&\Big((1+\norm{\nabla U^{n-1}}^2)\nabla U^n-(1+\norm{\nabla \widetilde u(t_{n-1})}^2)\nabla \widetilde u(t_n),\nabla\bar{\partial_t}\widehat\theta^n\Big)\notag\\
 &=\frac{(e^{\alpha k}-1)}{k}(\widehat\theta^n,\bar{\partial_t}\widehat\theta^n)+e^{\alpha t_n}(\bar{\partial_t}\eta^n,\bar{\partial_t}\widehat\theta^n)+e^{\alpha t_n}(u_t(t_n)-\bar{\partial_t}u(t_n),\bar{\partial_t}\widehat\theta^n)\notag\\
 &+e^{2\alpha t_n}\Big((1+\norm{\nabla u(t_n)}^2)\nabla u(t_n)-(1+\norm{\nabla\widetilde u(t_{n-1})}^2)\nabla\widetilde u(t_n),\nabla\bar{\partial_t}\widehat\theta^n\Big).\label{y1}
\end{align}
The second term on the left hand side of \eqref{y1}, can be bounde by
\begin{align}
 e^{\alpha t_n}\Big((1+\norm{\nabla U^{n-1}}^2)\nabla U^n-&(1+\norm{\nabla \widetilde u(t_{n-1})}^2)\nabla \widetilde u(t_n),\nabla\bar{\partial_t}\widehat\theta^n\Big)\notag\\
 &=(1+\norm{\nabla U^{n-1}}^2)(\nabla\widehat\theta^n,\nabla\bar{\partial_t}\widehat\theta^n)\notag\\
 &-e^{\alpha t_n}\Big((\norm{\nabla U^{n-1}}^2-\norm{\nabla\widetilde u(t_{n-1})}^2)\Delta u(t_n),\bar{\partial_t}\widehat\theta^n\Big).\label{y2}
\end{align}
The fourth term on the right hand side is bounded by
\begin{align}
 e^{\alpha t_n}\Big((1+\norm{\nabla u(t_n)}^2)\nabla u(t_n)-&(1+\norm{\nabla\widetilde u(t_n)}^2)\nabla\widetilde u(t_n),\nabla\bar{\partial_t}\widehat\theta^n\Big)\notag\\
 &=-e^{\alpha t_n}\Big((\norm{\nabla u(t_n)}^2-\norm{\nabla\widetilde u(t_{n})}^2)\Delta u(t_n),\bar{\partial_t}\widehat\theta^n\Big)\notag\\
 &-e^{\alpha t_n}\Big((\norm{\nabla\widetilde u(t_n)}^2-\norm{\nabla\widetilde u(t_{n-1})}^2)\Delta u(t_n),\bar{\partial_t}\widehat\theta^n\Big)\notag\\
 &\leq Che^{\alpha t_n}\norm{\nabla u(t_n)}\norm{\Delta u(t_n)}^2\norm{\bar{\partial_t}\widehat\theta^n}\notag\\
 &+e^{\alpha t_n}(\norm{\nabla u(t_n)}+\norm{\nabla u(t_{n-1})})(\norm{\nabla\widetilde u(t_n)-\nabla\widetilde u(t_{n-1})})\norm{\Delta u(t_n)}\norm{\bar{\partial_t}\widehat\theta^n}.\label{y3}
\end{align}
Therefore proceeding as in theorem \ref{thm5.6} we obtain
\begin{align}
 \norm{\bar{\partial_t}\widehat\theta^n}^2+e^{-\alpha k}(1+\norm{\nabla U^{n-1}}^2)\bar{\partial_t}\norm{\nabla\widehat\theta^n}^2&\leq C(\alpha)\Big(e^{-2\alpha k^*}\norm{\widehat\theta^n}^2+e^{-2\alpha k}e^{2\alpha t_n}\norm{\bar{\partial_t}\eta^n}^2\notag\\
 &+e^{-2\alpha k}e^{2\alpha t_n}\norm{u_t(t_n)-\bar{\partial_t}u(t_n)}^2\notag\\
 &+e^{-2\alpha k}(\norm{\nabla U^{n-1}}^2+\norm{\nabla u(t_{n-1})}^2)\norm{\Delta u(t_n)}^2\norm{\nabla\theta^{n-1}}^2\notag\\
 &+h^2e^{-2\alpha k}e^{2\alpha t_n}\norm{\nabla u(t_n)}^2\norm{\Delta u(t_n)}^4\notag\\
 &+\Big(e^{-2\alpha k}e^{2\alpha t_n}(\norm{\nabla u(t_n)}^2+\norm{\nabla u(t_{n-1})}^2)\notag\\
 &(\norm{\nabla\widetilde u(t_n)-\nabla\widetilde u(t_{n-1})})\norm{\Delta u(t_n)}\Big)\Big).\label{y4}
\end{align}
The rest of the proof is same as in Theorem \ref{thm5.6} which also uses the estimate in Theorem \ref{thm5.7}. This completes the rest of the proof.
\end{proof}
\section{Numerical Experiment}
In this section, we discuss fully discrete finite element formulation of \eqref{eq1}-\eqref{eq3} using modified backward Euler method. 
Now time variable is discritized by replacing the time derivative by difference quotient.
Let $k$ be the time step and $U^n$ be the approximation of $u(t)$ in $V_h$ at $t=t_n=nk.$ 
We now apply  modified backward Euler approximation to  \eqref{eq4}. 
\begin{example}
Here, we choose the right hand side function $f$ in such a way so that  the exact solution is $u=x(1-x)y(1-y)e^{-t}$ in $\Omega=(0, 1)\times(0, 1)$ and time $t=[0, 1]$, which satisfy the Dirichlet boundary condition.
\\In Table \ref{table:1}, the convergence rates are given for $t=1$. Observe that $\norm{\nabla u^n-\nabla U^n}$ is of order one as predicted by the theory. It is also observed numerically that the convergence rate for
$L^2$- error is of order $2$, but we still do not have a theory to back this claim.

\begin{table}[ht]
\centering
 \caption{Errors and convergence rate for modified backward Euler method}
 \begin{tabular}{c c c c c}
  \hline
  $h$ & $\norm{u(t_n)-U^n}_{L^2}$ & Conv. Rate & $\norm{u(t_n)-U^n}_{H^1}$ & Conv. Rate
 \\  \hline \hline
  $\frac{1}{2}$ & 0.002493 & & 0.010547 & \\ \hline
  $\frac{1}{4}$ & 0.000715 & 1.801077 & 0.006027 & 0.807341 \\ \hline
  $\frac{1}{8}$ & 0.000287 & 1.931655 & 0.003242 & 0.931121 \\ \hline
  $\frac{1}{16}$ & 0.000048 & 1.956815 & 0.001605 & 0.977676 \\ \hline
  $\frac{1}{32}$ & 0.000012 & 1.904400 & 0.000806 & 0.993114 \\
  \hline 
  \label{table:1}
 \end{tabular}
\end{table}

\begin{figure}
 \centering
 \includegraphics[height=7cm]{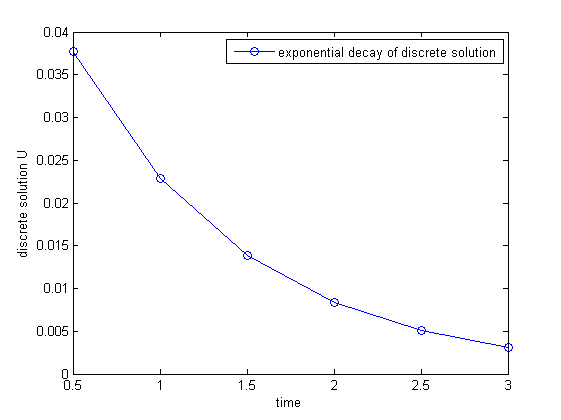}
 \caption{exponential decay of solution }
  \label{fig:1}
 \end{figure}

Since $f=O(e^{- t})$, it is further observed in Fig \ref{fig:1} that discrete solution $\norm{U}_{L^\infty(H^1)}$ decays exponentially as predicted by the theory.
\end{example}
\begin{example}
Here, we choose the right hand side function $f$ in such a way so that  the exact solution is $u=tsin(\pi x)sin(\pi y)$ in $\Omega=(0, 1)\times(0, 1)$ and time $t=[0, 1]$, which satisfy the Dirichlet boundary condition. 
\\In Table \ref{table:2}, the convergence rates are given for $t=1$. Observe that $\norm{\nabla u^n-\nabla U^n}$ is of order one as predicted by the theory.
\begin{table}[ht]
\centering
 \caption{Errors and convergence rate for modified backward Euler method}
 \begin{tabular}{c c c c c}
  \hline
  $h$ & $\norm{u(t_n)-U^n}_{L^2}$ & Conv. Rate & $\norm{u(t_n)-U^n}_{H^1}$ & Conv. Rate
 \\  \hline \hline
  $\frac{1}{2}$ & 0.091805 & & 0.470744 & \\ \hline
  $\frac{1}{4}$ & 0.024631 & 1.898054 & 0.268189 & 0.811695 \\ \hline
  $\frac{1}{8}$ & 0.006364 & 1.952374 & 0.141461 & 0.922836 \\ \hline
  $\frac{1}{16}$ & 0.001576 & 2.013456 & 0.072073 & 0.972865 \\ \hline
  $\frac{1}{32}$ & 0.000358 & 2.135151 & 0.036255 & 0.991270 \\
  \hline 
  \label{table:2}
 \end{tabular}
\end{table}
\end{example}

\begin{example}
 Now in this example we have taken right hand side $f=0$. We do not know the exact form of exact solution. We have chosen $u_0$(exact solution at $t=0$) as $u_0=x(1-x)y(1-y)sin(x+y)$.

 \begin{figure}
 \centering
 \includegraphics[height=7cm]{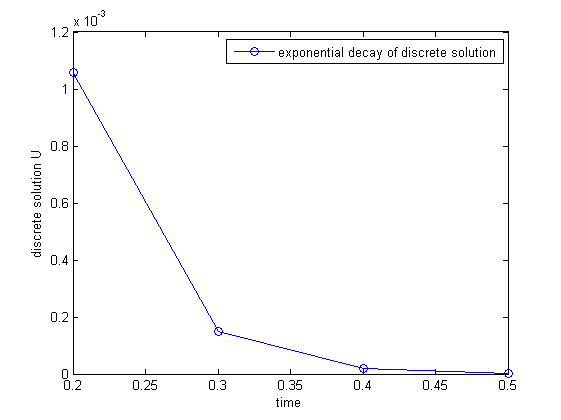}
 \caption{Exponential decay of solution when f=0}
  \label{fig:2}
 \end{figure}

Here we have observed in Fig \ref{fig:2} that discrete solution decays exponentially as time increase as predicted by the theory when $f=0$.
\end{example}
% \pagebreak
% % \newpage
% \newpage
 \bibliographystyle{amsplain}
 
%\bibliographystyle{plain}
%\bibliography{kirchhoff}

\end{document}